\documentclass{amsart}

\usepackage{graphicx,amssymb,xcolor,enumerate}
\usepackage{hyperref}
\hypersetup{
    bookmarks=true,         
    pdffitwindow=false,     
    pdfstartview={FitH},    
    colorlinks=false,       
}

\newtheorem{thm}{Theorem}[section]
\newtheorem{lem}[thm]{Lemma}
\newtheorem{prop}[thm]{Proposition}

\newtheorem{cor}[thm]{Corollary}

\theoremstyle{definition}

\theoremstyle{remark}
\newtheorem{rem}[thm]{Remark}

\theoremstyle{Fact}

\theoremstyle{Example}

\numberwithin{equation}{section}

 \makeatletter
\def\tagform@#1{\maketag@@@{\ignorespaces#1\unskip\@@italiccorr}}
\let\orgtheequation\theequation
\def\theequation{(\orgtheequation)}
\makeatother

  \allowdisplaybreaks
  \everymath{\displaystyle}

 \newcommand{\R}{\mathbb{R}}
 \newcommand{\T}{\mathbb{T}}

 \newcommand\Z{{\mathbb Z}}
 \newcommand\e{{\epsilon}}
 \newcommand\w{\omega}
 \newcommand\uu{\mathrm{u}}
 \newcommand\vv{\mathrm{v}}
 \newcommand\ww{\mathrm{w}}
 \newcommand\dx{\partial_x}
 
 \newcommand\dt{\partial_t}
 \newcommand\dr{\partial_r}
 \newcommand\da{\partial_a}
 \newcommand\dww{\partial_\ww}
 
\begin{document}

\title[Axisymmetric Hydrostatic Euler Equations]{Axisymmetric Flow of Ideal Fluid Moving in a Narrow Domain: a Study of the Axisymmetric
Hydrostatic Euler Equations}

\author{Robert M. Strain}
\thanks{R.M.S. was partially supported by the NSF grants DMS-1200747 and DMS-1500916.}

\author{Tak Kwong Wong}
\address{Department of Mathematics\\ University of Pennsylvania\\
David Rittenhouse Lab.\\
209 South 33rd Street \\
Philadelphia, PA 19104-6395, USA}
\email{strain@math.upenn.edu}
\email{takwong@math.upenn.edu}


\date{\today}



\begin{abstract}
In this article we will introduce a new model to describe the leading order behavior of an ideal and axisymmetric fluid moving in a very narrow domain. After providing a formal derivation of the model, we will prove the well-posedness and provide a rigorous mathematical justification for the formal derivation under a new sign condition. Finally, a blowup result regarding this model will be discussed as well.
\end{abstract}

\setcounter{tocdepth}{2} 
\maketitle
\tableofcontents

%
%
\section{Introduction}\label{s:Intro}
In various applications in meteorology, oceanography, atmospheric dynamics, blood flow and pipeline transport, the vertical or radial length scale of the underlying flow is usually small compared to the horizontal length scale. To study these problems, the standard approach is to apply the hydrostatic approximations. For example, when a two-dimensional ideal fluid moves in a fixed and very narrow channel, one can describe the leading order behavior of the fluid motion by the two-dimensional hydrostatic Euler equations, which can be formally derived by the hydrostatic limit \cite{Lio96} or the least action principle \cite{Bre08}. Under the local Rayleigh condition, the formal derivation of the two-dimensional hydrostatic Euler equations via the hydrostatic limit was rigorously justified in \cite{Gre99,Bre03,MW12}. Without the local Rayleigh condition, the formal derivation may not be valid \cite{Gre99,Gre00}. The local-in-time existence and uniqueness under the analyticity assumption \cite{KTVZ11}, the local Rayleigh condition \cite{Bre99,MW12}, or their combinations in different regions \cite{KMVW14} are also known, but the global-in-time existence is still open. Furthermore, for a general initial data, the two-dimensional hydrostatic Euler equations are somewhat ill-posed: see \cite{Ren09} for the linearized instability, and \cite{CINT15,Won15} for the formation of singularities.

In this paper, we study the leading order behavior of axisymmetric and ideal flows moving in a very narrow domain in three spatial dimensions. The prime objectives of this paper are as follows:
\begin{enumerate}[(i)]
\item to formally derive the axisymmetric hydrostatic Euler equations, which describe the leading order behavior of axisymmetric Euler flows moving in a thin tube, via the hydrostatic limit (see Subsection \ref{ss:FormalDerivation});
\item to introduce a \emph{new} sign condition (see inequality \eqref{e:sign} below), which is an analogue of the local Rayleigh condition in two spatial dimensions, for the axisymmetric hydrostatic Euler equations in three spatial dimensions;
\item to prove the well-posedness of the axisymmetric hydrostatic Euler equations under the new sign condition (see Theorem \ref{thm:Well-posed}, Sections \ref{s:APrioriEstimates} and \ref{s:Existence}, as well as Appendix \ref{app:3DSmoothVectorFields});
\item to provide a rigorous mathematical justification of the formal derivation for the axisymmetric hydrostatic Euler equations under the new sign condition (see Theorem \ref{thm:MathJustification} and Section \ref{s:MathematicalJustification});
\item to discuss the finite time blowup of smooth solutions of the axisymmetric hydrostatic Euler equations (see Theorem \ref{thm:FormationofSingularity} and Section \ref{s:FormationofSingularity}).
\end{enumerate}
To the best of our knowledge, these issues have not been studied in the literature.

The main difficulty for the study of axisymmetric hydrostatic Euler equations is the loss of the horizontal regularity. Similar to the two-dimensional hydrostatic Euler equations, the axisymmetric hydrostatic Euler equations are also derived by a singular limit process, called the hydrostatic limit in the literature. The hydrostatic limit/approximation usually simplifies the pressure term, but creates a loss of one horizontal derivative. Due to the regularity loss in the horizontal direction, standard energy methods do not apply in general.

The main novelty of this work is to introduce a new sign condition, which is an analogue of the local Rayleigh condition in two spatial dimensions. Under this sign condition, the horizontal regularity loss can be avoided by the nonlinear cancelation \eqref{e:NonlinearCancelation} below. As a result, the $H^s$ theory for the axisymmetric hydrostatic Euler equations can be established under this new sign condition by using the standard energy method. Furthermore, in order to simplify the computations, we will make use of new differential operator and dependent variables, which will be stated in \eqref{e:NewVar&Operator} below. Under these new differential operator and dependent variables, the vorticity system for the axisymmetric hydrostatic Euler equations is equivalent to that for the two-dimensional hydrostatic Euler equations in a certain sense. Therefore, the analysis in this paper is similar to that in \cite{Bre03,MW12}.

The rest of this paper is organized as follows. First of all, we will provide the formal derivation of the axisymmetric hydrostatic Euler equations, introduce the \emph{new} sign condition (i.e., inequality \eqref{e:sign}) and state our main results (i.e., Theorems \ref{thm:Well-posed}, \ref{thm:MathJustification} and \ref{thm:FormationofSingularity}) in Section \ref{s:FormalDerivation&MainResults}. In Section \ref{s:APrioriEstimates} we will derive a priori estimates, and apply these estimates to prove the uniqueness and stability. The existence will be shown in Section \ref{s:Existence}. Using the entropy method, we will provide a rigorous mathematical justification of the formal derivation for the axisymmetric hydrostatic Euler equations in Section \ref{s:MathematicalJustification}. Finally, the blowup result will be discussed in Section \ref{s:FormationofSingularity}. For the sake of self-containedness, we will also provide elementary proofs in Appendices \ref{app:BasicPropertiesforAHEE}-\ref{app:3DSmoothVectorFields}.

Let us end this introduction by commenting on our notation. Throughout this paper, all constants with or without subscript(s) may be different in different lines. Unless mentioned otherwise, a constant with subscript(s) illustrates the dependence of the constant, for example, $C_{s,\sigma}$ is a constant depending on $s$ and $\sigma$ only.

%
%
\section{Formal Derivation and Main Results}\label{s:FormalDerivation&MainResults}
In this section, we will first provide a heuristic derivation of the axisymmetric hydrostatic Euler equations in Subsection \ref{ss:FormalDerivation}, and then state the main results of this paper in Subsection \ref{ss:MainResults}.

%
%
\subsection{Formal Derivation}\label{ss:FormalDerivation}
The aim of this subsection is to formally derive the axisymmetric hydrostatic Euler equations via a rescaling limit. This rescaling limit, called the hydrostatic limit, can be described as follows.

An ideal fluid moving in a periodic and narrow domain\footnote{In order to simplify our presentation, we only consider
  a periodic tube (i.e., $X\in\T:=\R/\Z$), but one could consider an infinite tube (i.e., $X \in \R$) or other physical
  domains. } $\Omega_\e:=\{(X,Y,Z);\;X\in \T:=\R/\Z,\;Y^2+Z^2<\e^2\}$
is governed by the usual three-dimensional incompressible Euler equations:
%
%
\begin{equation}\label{e:3DEuler}
\left\{\begin{aligned}
\dt \vec{U}_\e +\vec{U}_\e  \cdot \nabla \vec{U}_\e & := - \nabla P_\e && \text{ in }
(0,T) \times \Omega_\e\\
{\rm div}\; \vec{U}_\e & = 0 && \text{ in } (0,T) \times \Omega_\e\\
\vec{U}_\e  \cdot \hat{n}|_{\partial \Omega_\e} & = 0\\
\vec{U}_\e |_{t=0} & = \vec{U}_{\e 0},
\end{aligned}\right.
\end{equation}
where $\vec{U}_\e$ is an unknown velocity field, $P_\e$ is an unknown scalar pressure,
$\vec{U}_{\e 0}$ is a given initial velocity field, and $\hat{n}$ is the outward unit normal vector to the boundary $\partial
\Omega_\e$.

Now, let us assume that the underlying flow is axisymmetric without swirl. In other words, using the cylindrical
coordinates
\[\begin{aligned}
X&:=X,& Y&:=R \cos \Theta,\text{ and}& Z &:= R \sin \Theta,
\end{aligned}\]
we can express the velocity field $\vec{U}_\e$ and the pressure $P_\e$ as follows:
%
%
\begin{equation}\label{e:Axisymmetric}
\left\{\begin{aligned}
\vec{U}_\e & := U^X_\e (t,X,R) e_X + U^R_\e (t,X,R) e_R\\
P_\e & := P_\e (t,X,R),
\end{aligned}\right.
\end{equation}
where $e_X$ and $e_R$ are the unit vectors in the $X$ and $R$ directions respectively. Under the
assumption \eqref{e:Axisymmetric}, the usual incompressible Euler equations \eqref{e:3DEuler} become
the axisymmetric Euler equations: for $(t,X,R) \in (0,T) \times \T \times (0,\e)$,
%
%
\begin{equation}\label{e:AEE}
\left\{\begin{aligned}
\dt U^X_\e + U^X_\e \partial_X U^X_\e + U^R_\e \partial_R U^X_\e & = - \partial_X P_\e\\
\dt U^R_\e + U^X_\e \partial_X U^R_\e + U^R_\e \partial_R U^R_\e & = - \partial_R P_\e\\
\partial_X U^X_\e + \frac{1}{R} \partial_R (R U^R_\e) & = 0\\
 U^R_\e |_{R=\e} & = 0\\
 U^X_\e|_{t=0} & = U^X_{\e 0}.
\end{aligned}\right.
\end{equation}
Here, we only have to impose the initial condition for $U^X_\e$ because that for $U^R_\e$ can
be uniquely determined by using $\eqref{e:AEE}_3$-$\eqref{e:AEE}_5$.

In order to study the leading order behavior of the flow as the thickness $\e$ goes to $0^+$, the standard approach is to rescale the physical domains into a uniform domain. More precisely, we apply the rescaling
\begin{equation}\label{e:rescaling}
\left\{\begin{aligned}
x&:=X,&   r&:= \frac{R}{\e},\\
U^X_\e(t,X,R)&:=u^x_\e\left(t,X,\frac{R}{\e}\right),& U^R_\e(t,X,R)&:= \e u^r_\e \left(t,X,\frac{R}{\e}\right),\\
P_\e (t,X,R) &:= p_\e \left(t,X,\frac{R}{\e}\right)
\end{aligned}\right.
\end{equation}
to rewrite the system \eqref{e:AEE} as the axisymmetric rescaled Euler equations: for $(t,x,r) \in (0,T) \times
\T \times (0,1)$,
\begin{equation}\label{e:AREE}
\left\{\begin{aligned}
\dt u^x_\e + u^x_\e \dx u^x_\e + u^r_\e \dr u^x_\e & = - \dx p_\e \\
\e^2 (\dt u^r_\e + u^x_\e \dx u^r_\e + u^r_\e \dr u^r_\e) & = - \dr p_\e \\
\dx u^x_\e +\frac{1}{r} \dr (r u^r_\e) & = 0 \\
u^r_\e|_{r=1} & = 0\\
u^x_\e|_{t=0} & = u^x_{\e 0},
\end{aligned}\right.
\end{equation}
where $(u^x_\e, u^r_\e)$ and $p_\e$ are unknowns, and $u^x_{\e 0}$ is the given initial horizontal velocity.

Formally, if $(u^x_\e, u^r_\e, p_\e)$ converges to $(u^x, u^r, p)$ as $\e$ goes to $0^+$, then $(u^x, u^r,p)$
will satisfy the axisymmetric hydrostatic Euler equations: for $(t,x,r) \in (0,T) \times \T \times (0,1)$,
\begin{equation}\label{e:AHEE2.1}
\left\{\begin{aligned}
\dt u^x+u^x\dx u^x + u^r \dr u^x & = -\dx p \\
\dr p & = 0 \\
\dx u^x + \frac{1}{r} \dr (ru^r) & = 0 \\
u^r |_{r=1} & = 0\\
u^x|_{t=0} & = u^x_0.
\end{aligned}\right.
\end{equation}
It is worth noting that equation $\eqref{e:AHEE2.1}_2$ is equivalent to the fact that the scalar pressure $p$ is
independent of $r$, i.e., $p:=p(t,x)$. Therefore, system \eqref{e:AHEE2.1} is equivalent to system \eqref{e:AHEE}
below provided that $p$ is assumed to be independent of $r$.

\emph{The system \eqref{e:AHEE2.1} formally describes the leading order behavior of axisymmetric and ideal flows
moving in the narrow domain $\Omega_\e$.} The above limiting process is called the hydrostatic limit.

Let us end this subsection by explaining why the vorticity $\w$ for the axisymmetric hydrostatic Euler equations \eqref{e:AHEE2.1} (or equivalently, \eqref{e:AHEE}) is $\dr u^x$ as follows:
%
%
\begin{rem}[Vorticity for the Axisymmetric Hydrostatic Euler Equations]\label{rem:vorticityforAHEE}
It is well-known that the vorticity for the axisymmetric (without swirl) velocity field $\eqref{e:Axisymmetric}_1$ is just $-(\partial_R U^X_\e - \partial_X U^R_\e)e_\Theta$, where $e_\Theta$ is the unit vector in the $\Theta$ direction. Applying the rescaling \eqref{e:rescaling}, we have
\[
\partial_R U^X_\e - \partial_X U^R_\e = \frac{1}{\e} \left( \dr u^x_\e - \e^2 \dx u^r_\e \right).
\]
As $\e \to 0^+$, the quantity $\dr u^x_\e - \e^2 \dx u^r_\e$ converges to $\dr u^x$ formally. Therefore, we denote the vorticity for the axisymmetric hydrostatic Euler equations \eqref{e:AHEE2.1} (or equivalently, \eqref{e:AHEE}) as the scalar quantity $\dr u^x$, which is corresponding to the leading order term of the vorticity as $\e \to 0^+$.
\end{rem}

%
%
\subsection{Main Results}\label{ss:MainResults}
In this subsection we will introduce the axisymmetric hydrostatic Euler equations, the new sign
condition and our main results.

According to the formal derivation in Subsection \ref{ss:FormalDerivation}, the leading order behavior of an
axisymmetric (without swirl) ideal flow moving in a periodic and narrow channel can be described by the
axisymmetric hydrostatic Euler equations: for $(t,x,r) \in (0,T) \times \T \times (0,1)$,
\begin{equation}\label{e:AHEE}
\left\{\begin{aligned}
\dt u^x + u^x \dx u^x + u^r \dr u^x & = - \dx p \\
\dx u^x + \frac{1}{r} \dr(ru^r) & = 0 \\
u^r|_{r=1} & = 0\\
u^x|_{t=0} & = u^x_0,
\end{aligned}\right.
\end{equation}
where the axisymmetric (without swirl) velocity field $\vec{u}:=u^x(t,x,r)e_x+u^r(t,x,r)e_r$ and the
scalar pressure $p:=p(t,x)$ are unknowns, $u^x_0:=u^x_0(x,r)$ is the given initial horizontal velocity,
$e_x$ and $e_r$ are the unit vectors in the horizontal (i.e., $x$) and radial (i.e., $r$) directions respectively.

Regarding the axisymmetric hydrostatic Euler equations \eqref{e:AHEE}, there are at least three fundamental problems:
\begin{enumerate}[(i)]
\item the local-in-time well-posedness;
\item the mathematical justification of the formal derivation; and
\item the formation of singularity.
\end{enumerate}
In this paper we address all these problems. More specifically, we will first introduce a new sign condition, under
which we will prove the local-in-time well-posedness of $H^s$ solutions to \eqref{e:AHEE} and justify the formal derivation in the
$L^2$ sense. Furthermore, we will also show that for a certain class of initial data, smooth solutions to \eqref{e:AHEE} blow up in finite
time.

When studying the well-posedness and the mathematical justification of the formal derivation (i.e., problems (i) and
(ii) above), one encounter the following

\begin{quotation}
\underline{\textbf{Structural Difficulty:}}\\*
\indent
The radial velocity component $u^r =-\frac{1}{r} \dr^{-1} (r\dx u^x)$ creates a loss of one $x$-derivative, so the standard
energy methods typically fail.
\end{quotation}

This structural difficulty can be overcome if we consider the problems under the following sign
condition:
\begin{equation}\label{e:sign}
\frac{1}{r}\dr \left(\frac{1}{r} \dr u^x\right) \ge \sigma >0
\end{equation}
for some constant $\sigma$. To the best of our knowledge, the sign condition \eqref{e:sign} is new, and analogous to the local Rayleigh condition
for the two dimensional flow. Under the sign condition \eqref{e:sign}, we can avoid the structural difficulty by eliminating the
problematic term during the estimation. This elimination is based on the following nonlinear cancelation for the
axisymmetric hydrostatic Euler equations \eqref{e:AHEE}: defining the vorticity $\w := \dr u^x$, for any $k = 0$, $1$, $2$, $\cdots,$
\begin{equation}\label{e:NonlinearCancelation}
\int_\T \int^1_0 \dx^k u^r \dx^k \w \; rdrdx = \frac{1}{2} \int_\T \int^1_0 \dx \left( |\dx^k u^x|^2\right)\; rdrdx = 0.
\end{equation}
This nonlinear cancelation is analogous to the nonlinear cancelation (2.3) stated in \cite{MW12} for the two-dimensional
flow as well.

Regarding the well-posedness (i.e., problem (i) above), we will solve the axisymmetric hydrostatic Euler equations
\eqref{e:AHEE} as long as the quantity $L_r u^x := \frac{1}{r} \dr u^x$ belongs to the function space
\begin{equation}\label{e:DefnofH^s_Lsigma}
H^s_{L,\sigma} := \left\{ \ww:\T \times (0,1)\to\R;\; \|\ww\|_{H^s_L} < + \infty \text{ and } \sigma \le L_r \ww \le \frac{1}{\sigma}\right\}
\end{equation}
for some integer $s\ge 4$ and constant $\sigma \in (0,1)$, where the operator $L_r:= \frac{1}{r} \dr$, and the norm
$\|\cdot\|_{H^s_L}$ is defined by
\begin{equation}\label{e:defnH}
\|\ww\|^2_{H^s_L} := \sum_{|\alpha|\le s} \|\dx^{\alpha_1} L_r^{\alpha_2} \ww\|^2_{L^2(rdrdx)}
:= \sum_{|\alpha| \le s} \int_\T \int^1_0 |\dx^{\alpha_1} L_r^{\alpha_2} \ww|^2 \; rdrdx.
\end{equation}
Roughly speaking\footnote{It is worth noting that if $\ww\in H^s_L$, then the upper bound $L_r \ww \le \frac{1}{\sigma}$ can always be guaranteed by the Sobolev inequality \eqref{e:sobolevinq} provided that $\sigma >0$ is chosen to be sufficiently small. In particular, there is no mathematical or physical meaning for the upper bound $\frac{1}{\sigma}$. In other words, one can choose other upper bounds instead. We just require the $L^\infty$ boundedness for $L_r \ww$ so that the $H^s_L$ norm \eqref{e:defnH} is equivalent to the weighted $H^s_L$ energy \eqref{e:defnWH}.}, we require that the horizontal velocity component $u^x$ satisfies the sign condition \eqref{e:sign}, and the quantity $L_r u^x := \frac{1}{r}
\dr u^x$ belongs to the function space
\begin{equation}\label{e:DefnofH^s_L}
H^s_L := \{\ww:\T \times (0,1)\to\R;\; \|\ww\|_{H^s_L} < + \infty\}.
\end{equation}

More precisely, we will prove the following main theorem:
%
%
\begin{thm}[Well-posedness for the Axisymmetric Hydrostatic Euler Equations \eqref{e:AHEE}]\label{thm:Well-posed}
For any integer $s\ge 4$ and constant $\sigma > 0$, if the given initial data $u^x_0$ satisfies the following two hypotheses:
\begin{enumerate}[(i)]
\item (regularity and sign conditions) $\ww_0 := L_r u^x_0 := \frac{1}{r} \dr u^x_0 \in H^s_{L,2\sigma}$,
\item (compatibility condition) $\int^1_0 u^x_0 \; r dr \equiv \lambda \equiv$ constant,
\end{enumerate}
then there exist a time $T := T(s,\sigma, \|\ww_0\|_{H^s_L}) > 0$, a unique axisymmetric velocity field $\vec{u}:= u^x e_x
+ u^r e_r$, and a unique (up to an additional function of $t$) scalar pressure $p$ such that
\begin{enumerate}[(i)]
\item $(\vec{u},p)$ solves the axisymmetric hydrostatic Euler equations \eqref{e:AHEE} classically,
\item $\ww := L_r u^x := \frac{1}{r} \dr u^x \in C([0,T]; H^s_{L,\sigma}) \cap C^1 ([0,T]; H^{s-1}_L)$, and
\item $\int^1_0 u^x \; r dr \equiv \int^1_0 u^x_0 \; r dr \equiv \lambda$.
\end{enumerate}
Furthermore, the system \eqref{e:AHEE} is stable in the following sense: for any integer $s'\in[0,s)$, and any two solutions $(\vec{u_1},p_1):=(u_1^x e_x+ u_1^r e_r,p_1)$ and  $(\vec{u_2},p_2):=(u_2^x e_x+ u_2^r e_r,p_2)$ of system \eqref{e:AHEE}, if $\ww_i := L_r u^x_i:=\frac{1}{r} \dr u^x_i \in C([0,T]; H^s_{L,\sigma}) \cap C^1 ([0,T]; H^{s-1}_L)$ and $\int^1_0 u^x_i \; r dr \equiv \lambda\equiv$ constant for $i=1$, $2$, then there exists a constant $C_{s,s',\sigma,T,M}>0$ such that
\begin{equation}\label{e:StabilityEsts}
\|\ww_1-\ww_2\|_{C([0,T];H^{s'}_L)} \le C_{s,s',\sigma,T,M} \|(\ww_1-\ww_2)|_{t=0}\|_{L^2(rdrdx)}^{1-\frac{s'}{s}},
\end{equation}
where $M:=\max\left\{\sup_{[0,T]}\|\ww_1\|_{H^{s}_L},\sup_{[0,T]}\|\ww_2\|_{H^{s}_L}\right\}$, and the $H^s_L$ norm is defined by \eqref{e:defnH}.

Moreover, for any $t\in [0,T]$, the unique axisymmetric velocity field $\vec{u}(t)$ is also a $C^{s-2}$ vector field in three spatial dimensions.
\end{thm}

%
%
\begin{rem}
\hfil
\begin{enumerate}[(i)]
\item (Intuition behind Norm \eqref{e:defnH}) Indeed, the norm \eqref{e:defnH} is natural for axisymmetric functions. First of all, viewing all axisymmetric functions as a special class of functions defined in a three-dimensional domain, one can see that the volume element is just a $2\pi$ multiple of $rdrdx$ because axisymmetric functions are independent of the azimuthal coordinate $\theta$. Therefore, $rdrdx$ is an appropriate and natural measure for the axisymmetric physical quantities. Regarding the measure $rdr$, the differential operator that can guarantee the fundamental theorem of calculus \eqref{e:FTCforL_r&rdr} is $L_r := \frac{1}{r}\dr$, so it is also natural to measure the radial regularity for axisymmetric functions by using the differentiation with respect to $L_r$ instead of $\dr$.
\item (Structure of the Proof) The proof of Theorem \ref{thm:Well-posed} is long, and will be separated into different parts: in Section \ref{s:APrioriEstimates} we will derive a priori estimates and prove the uniqueness and stability; in Section \ref{s:Existence} we will show the existence by using an approximate scheme (see Subsection \ref{ss:ExistenceviaApproximateSystem}) or a reduction argument (see Subsection \ref{ss:ExistenceviaReduction}); finally, we will verify the $C^{s-2}$ regularity of $\vec{u}(t)$ in Appendix \ref{app:3DSmoothVectorFields}.
\item (Solvability in Non-periodic Domains) In Theorem \ref{thm:Well-posed} we solve the axisymmetric hydrostatic Euler equations \eqref{e:AHEE} in the $x$-periodic domain $\T\times (0,1):=\{ (x,r); \; x\in\T,\; r\in(0,1) \}$. However, it is just for presentation convenience. For example, one may apply the result in Theorem \ref{thm:Well-posed}, the finite speed of propagation in the $x$-direction, and an argument similar in \cite{KMVW14} to solve the axisymmetric hydrostatic Euler equations \eqref{e:AHEE} in other non-periodic domain like $\R\times (0,1)$. For this purpose, one should also consider the uniformly local $H^s_L$ spaces:
\[
H^s_{L,uloc} := \left\{ \ww:\R \times (0,1)\to\R;\; \sup_{l\in\Z}\|\ww\|_{H^s_L ([l-1,l+1])} < \infty\right\},
\]
    where $\|\ww\|_{H^s_L ([l-1,l+1])}^2:= \sum_{|\alpha| \le s} \int^{l+1}_{l-1} \int^1_0 |\dx^{\alpha_1} L_r^{\alpha_2} \ww|^2 \; rdrdx$. We leave the details to the interested reader.
\end{enumerate}
\end{rem}

Regarding the mathematical justification of the formal derivation (i.e., problem (ii) above), we will show the
following
%
%
\begin{thm}[Mathematical Justification of the Formal Derivation of the Axisymmetric Hydrostatic Euler Equations
\eqref{e:AHEE}]\label{thm:MathJustification} For any $\e > 0$, let $(u_\e^x, u_\e^r, p_\e)$ and $(u^x, u^r, p)$ be smooth solutions
to the axisymmetric rescaled Euler equations
\eqref{e:AREE} and the axisymmetric hydrostatic Euler equations \eqref{e:AHEE} respectively. Assume that $u^x$
satisfies the sign condition \eqref{e:sign} for some constant $\sigma > 0$. If there exist constants $C_0>0$ and
$\beta \in(0,4]$ such that
\[
\left. \int_\T \int^1_0 \left\{|u^x_\e - u^x|^2 + \e^2|u^r_\e-u^r|^2 + |\ww_\e-\ww|^2\right\}\;rdrdx\right|_{t=0} \le C_0 \e^\beta,
\]
where $\ww_\e := \frac{1}{r} \dr u^x_\e - \frac{\e^2}{r} \dx u^r_\e$ and $\ww := \frac{1}{r} \dr u^x$, then
for all $t \in [0,T]$,
\[
\int_\T \int^1_0 \left\{|u^x_\e-u^x|^2+\e^2|u^r_\e-u^r|^2+|\ww_\e-\ww|^2\right\} \;rdrdx
\le \widetilde{C} \e^\beta,
\]
where the constant $\widetilde{C}$ depends only on $\sigma$, $u^x$, $u^r$, $\ww$, $C_0$ and $T$, but not on
$\e$ nor $(u^x_\e, u^r_\e, p_\e)$.
\end{thm}

The proof of Theorem \ref{thm:MathJustification} will be provided in Section \ref{s:MathematicalJustification}.

%
%
\begin{rem}
\hfil
\begin{enumerate}[(i)]
\item (Existence of the Axisymmetric Rescaled Euler Equations \eqref{e:AREE}) Since the axisymmetric rescaled Euler equations \eqref{e:AREE} is obtained directly from the axisymmetric Euler equations \eqref{e:AEE} by applying the rescaling \eqref{e:rescaling}. Therefore, the global-in-time existence to the axisymmetric rescaled Euler equations \eqref{e:AREE} follows directly from that to the axisymmetric Euler equations \eqref{e:AEE}. Indeed, the global-in-time existence of regular solutions (e.g., $\vec{U}_{\e}:=U^X_\e e_X + U^R_\e e_R \in C([0,\infty);H^s)\cap C^1([0,\infty);H^{s-1})$ for $s>\frac{5}{2}$) to the axisymmetric Euler equations \eqref{e:AEE} is classical; see \cite{UI68} and \cite{Maj86} for the case of $\R^3$, as well as \cite{SY94} for the case of bounded domains in $\R^3$ for instance.
\item ($L^2$ Mathematical Justification) Since the local-in-time solution for the axisymmetric hydrostatic Euler equations \eqref{e:AHEE} obtained in Theorem \ref{thm:Well-posed} fulfills the hypotheses stated in Theorem \ref{thm:MathJustification}, we justify mathematically rigorously in the $L^2$ sense that the solution obtained in Theorem \ref{thm:Well-posed} provides a reasonable approximation of the leading order behavior of axisymmetric and incompressible fluids moving in a very thin tube.
\end{enumerate}
\end{rem}

Regarding the formation of singularity (i.e., problem (iii) above), we will prove the following
%
%
\begin{thm}[Finite Time Blowup for Smooth Solutions to the Axisymmetric Hydrostatic Euler Equations
\eqref{e:AHEE}]\label{thm:FormationofSingularity}
Consider the axisymmetric hydrostatic Euler equations \eqref{e:AHEE} in the physical domain $\R\times (0,1):=\left\{ (x,r); \; x\in\R \text{ and } 0<r<1 \right\}$ instead of $\T\times (0,1)$. Let the initial data $u^x_0:=u^x_0 (x,r)$ satisfy the following property: there exist a fixed position $\hat{x}\in\R$ and a constant horizontal velocity $\widehat{u^x}\in\R$ such that
\begin{equation}\label{e:BlowupHypothesis}
\left\{
\begin{aligned}
u^x_0 (\hat{x},r) &\equiv \widehat{u^x} \quad\text{for all }r\in [0,1]\\
\text{either}\quad\dx L_r u^x_0 (\hat{x},0) &= 0 \quad\text{or}\quad\dx L_r u^x_0 (\hat{x},1) = 0\\
\dx L_r^2 u^x_0(\hat{x},r) &< 0 \quad\text{for all }r\in (0,1),
\end{aligned}\right.
\end{equation}
where $L_r:=\frac{1}{r}\dr$. Then there exist a finite time $T>0$ such that if a solution $(\vec{u},p):=(u^x e_x + u^r e_r,p)$ to the axisymmetric hydrostatic Euler equations \eqref{e:AHEE} remains smooth in the time interval $[0,T)$, then
\begin{equation}\label{e:FiniteTimeBlowupu^xdxu^xdxp}
\lim_{t\to T^-} \max\left\{ \|u^x(t)\|_{L^\infty}, \|\dx u^x(t)\|_{L^\infty}, \|\dx p(t)\|_{L^\infty} \right\}=\infty.
\end{equation}
\end{thm}

The proof of Theorem \ref{thm:FormationofSingularity} will be provided in Section \ref{s:FormationofSingularity}.

%
%
\begin{rem}
\hfil
\begin{enumerate}[(i)]
\item (Periodic Domain vs. Infinite Tube) In Theorems \ref{thm:Well-posed} and \ref{thm:MathJustification}, we consider the axisymmetric hydrostatic Euler equations \eqref{e:AHEE} in the periodic physical domain $\T\times (0,1)$ for simplicity. However, in Theorem \ref{thm:FormationofSingularity} we consider the fluids in the infinite tube $\R\times (0,1)$ instead. It is worth noting that the result stated in Theorem \ref{thm:FormationofSingularity} also applies to the fluids in the periodic domain $\T\times (0,1)$ because one may consider the solutions in the periodic domain $\T\times (0,1)$ as periodic flows in the infinite tube $\R\times (0,1)$.
\item (Blowup) The finite time blowup \eqref{e:FiniteTimeBlowupu^xdxu^xdxp} is in $\|u^x\|_{L^\infty}$, $\|\dx u^x\|_{L^\infty}$, or $\|\dx p\|_{L^\infty}$. The first and last cases correspond to the infinite horizontal velocity and pressure gradient respectively; they are non-physical in a certain sense. The second case corresponds to the formation of singularity.
\end{enumerate}
\end{rem}

%
%
\section{A Priori Estimates}\label{s:APrioriEstimates}
The aim of this section is to derive the a priori energy estimates for proving the well-posedness of the axisymmetric hydrostatic Euler equations \eqref{e:AHEE}. We will derive these estimates in following three steps:
\begin{enumerate}[(i)]
\item in order to simplify our computations, the axisymmetric hydrostatic Euler equations \eqref{e:AHEE} will be rewritten as the vorticity formulation in terms of the new differential operator and dependent variables \eqref{e:NewVar&Operator} in Subsection \ref{ss:VorticityFormulationinNewVariables};
\item the weighted $H^s_L$ energy estimates will be derived in Subsection \ref{ss:WeightedEnergyEstimates}; and
\item an $L^2$ comparison principle as well as its applications to uniqueness and stability will be provided in Subsection \ref{ss:L^2ComparisonPrinciple}.
\end{enumerate}

%
%
\subsection{Vorticity Formulation in New Variables}\label{ss:VorticityFormulationinNewVariables}
In this subsection we will first rewrite the axisymmetric hydrostatic Euler equations \eqref{e:AHEE} in the vorticity formulation so that we can avoid handling the scalar pressure $p$, which is a non-local quantity. After that, applying new dependent variables $\uu$, $\vv$, $\ww$ and differential operator $L_r$, we will further rewrite the vorticity formulation into a better form.

Let us begin by introducing the vorticity formulation for the axisymmetric hydrostatic Euler equations \eqref{e:AHEE} as follows.

In order to derive estimates for system \eqref{e:AHEE}, one may encounter the technical difficulty that the scalar pressure $p:=p(t,x)$ is a non-local quantity, which and the horizontal velocity component $u^x$ are related by an integral expression:
\[
\dx p = - 2 \;\dx \int^1_0 \left|u^x\right|^2 \; rdr.
\]
Inspired by the usual Euler equations, we can avoid handling the scalar pressure $p$ by considering the vorticity formulation. For the axisymmetric hydrostatic Euler equations \eqref{e:AHEE}, the vorticity\footnote{See Remark \ref{rem:vorticityforAHEE} for the explanation.} $\w$ is defined by $\w:=\dr u^x$. Differentiating the momentum equation $\eqref{e:AHEE}_1$ with respect to $r$ and using the incompressibility condition $\eqref{e:AHEE}_2$, we obtain the vorticity equation
\[
\dt\w + u^x \dx\w + u^r\dr\w = \frac{u^r}{r}\w.
\]
Or equivalently,
\begin{equation*}
\left(\dt + u^x \dx + u^r\dr\right)\left(\frac{\w}{r}\right) = 0.
\end{equation*}

Therefore, instead of studying equations \eqref{e:AHEE} directly, we consider the following vorticity system: for $(t,x,r)\in (0,T)\times \T \times (0,1)$,
\begin{equation}\label{e:VFAHEE}
\left\{\begin{aligned}
\dt \w+u^x \dx \w + u^r \dr \w & = \frac{u^r}{r} \w \\
u^x & = - \frac{1}{r} \dr \mathcal{A} \left(\frac{\w}{r}\right) \\
u^r & = \frac{1}{r} \dx \mathcal{A} \left(\frac{\w}{r}\right) \\
\w|_{t=0} & = \w_0 := \dr u^x_0,
\end{aligned}\right.
\end{equation}
where the Dirichlet solver $\mathcal{A}$ is defined by
\begin{equation}\label{e:DefnofA}
\mathcal{A}(f)(t,x,r) := -\frac{1}{4} \int^1_0 \{|r^2 - \rho^2| -r^2 - \rho^2  + 2r^2 \rho^2\} f(t,x,\rho) \; \rho
d\rho,
\end{equation}
which is the unique solution to
\[\left\{\begin{aligned}
- \frac{1}{r} \dr \left(\frac{1}{r} \dr \mathcal{A} (f) \right) & = f& \text{for }(x,r)\in\T\times (0,1)\\
\mathcal{A} (f) |_{r=0,1} & = 0.
\end{aligned}\right.\]
\begin{rem}[Biot-Savart Formulae]
The equations $\eqref{e:VFAHEE}_2$-$\eqref{e:VFAHEE}_3$ can be considered as the Biot-Savart formulae for the
axisymmetric hydrostatic Euler equations \eqref{e:AHEE}.
\end{rem}

According to the equivalence Lemma \ref{lem:Equivalence}, the axisymmetric hydrostatic Euler equations
\eqref{e:AHEE} and the vorticity system \eqref{e:VFAHEE} are equivalent provided that the pole condition
\begin{equation}\label{e:PoleConditnfor u^r}
\lim_{r\to 0^+} r u^r(t,x,r) = 0,
\end{equation}
and the compatibility condition
\begin{equation}\label{e:CompatibilityCondition}
\int^1_0 u^x \; rdr \equiv 0
\end{equation}
hold. Indeed, the pole condition \eqref{e:PoleConditnfor u^r} is just a minor and technical assumption because we are looking for regular solutions. Furthermore, without loss of generality, we can always assume the compatibility condition \eqref{e:CompatibilityCondition}
because of the following
\begin{rem}[Compatibility Condition \eqref{e:CompatibilityCondition}]\label{rem:CompatibilityCondition}
Using the incompressibility condition $\eqref{e:AHEE}_2$, the boundary condition $\eqref{e:AHEE}_3$ and the pole condition \eqref{e:PoleConditnfor u^r}, we know that at the initial time $t=0$, the integral $\int^1_0 u^x_0\;rdr$ is independent of $x$, and hence, must be a constant. Indeed, this constant is the $\lambda$ stated in Theorem \ref{thm:Well-posed}. Therefore, we can always fulfill the compatibility condition \eqref{e:CompatibilityCondition} at the initial time $t=0$ by applying a suitable Galilean transformation. Due to Lemma \ref{lem:ConservationofAveu^x}, the compatibility condition \eqref{e:CompatibilityCondition} also holds for all later times $t>0$ since it does initially.
\end{rem}

From now on until the end of Section \ref{s:Existence}, we will always assume \eqref{e:PoleConditnfor u^r} and \eqref{e:CompatibilityCondition}, and
consider the vorticity system \eqref{e:VFAHEE} instead of the axisymmetric hydrostatic Euler equations \eqref{e:AHEE}.
Now, in order to simplify our computations, we are going to further rewrite the vorticity system \eqref{e:VFAHEE} in terms
of \emph{new} differential operator and dependent variables as follows.

Using
\begin{equation}\label{e:NewVar&Operator}
\begin{aligned}
L_r &:=\frac{1}{r}\dr,& \uu &:=u^x,& \vv &:= r u^r,\text{ and}& \ww &:= \frac{\w}{r}=L_r \uu,
\end{aligned}
\end{equation}
we can rewrite the vorticity system \eqref{e:VFAHEE} as follows: for $(t,x,r) \in (0,T) \times \T \times (0,1)$,
\begin{equation}\label{e:VFAHEENV}
\left\{\begin{aligned}
\dt \ww + \uu \dx \ww + \vv L_r \ww & = 0 \\
\uu & = - L_r \mathcal{A} (\ww) \\
\vv & = \dx \mathcal{A} (\ww) \\
\ww|_{t=0} & = \ww_0 := \frac{\w_0}{r}
\end{aligned}\right.
\end{equation}
where the Dirichlet solver $\mathcal{A}$ is defined by \eqref{e:DefnofA}. In other words, $\mathcal{A}(\ww)$ is
the unique solution to
\[
\left\{\begin{aligned}
-L_r^2 \mathcal{A}(\ww) & = \ww && \text{ for } (x,r) \in \T \times (0,1)\\
\mathcal{A} (\ww) |_{r=0,1} & = 0.
\end{aligned}\right.
\]
Indeed, there are at least three advantages for considering the new vorticity system \eqref{e:VFAHEENV} instead of the original vorticity system
\eqref{e:VFAHEE}:
\begin{enumerate}[(i)]
\item in terms of $L_r$, $\uu$ and $\vv$, the incompressibility condition $\eqref{e:AHEE}_2$ becomes a simpler
form:
\begin{equation}\label{e:IncompressibilityConditnindx&Lr}
\dx \uu + L_r \vv = 0;
\end{equation}
\item the system \eqref{e:VFAHEENV} is ``almost'' the same as the vorticity system for the two-dimensional hydrostatic Euler equations if
we replace $L_r$ by $\da$, see Subsection \ref{ss:ExistenceviaReduction} for the details;
\item the unknowns $\uu$, $\vv$ and $\ww$ are better quantities in the following sense: if $\vec{u} := u^x e_x
+u^r e_r$ is a smooth vector field in three spatial dimensions, then the horizontal velocity component $u^x$
is a smooth function, but the quantities $u^r$ and $\w := \dr u^x$ may have singularities at the symmetric axis
$r=0$. However, one may check that all three quantities $\uu$, $\vv$ and $\ww$ are smooth even at the symmetric axis
$r=0$. As a result, we can estimate these new quantities
relatively easier.
\end{enumerate}
Due to the above three advantages, we will study the vorticity system \eqref{e:VFAHEENV} instead of the original vorticity system
\eqref{e:VFAHEE} in Subsections \ref{ss:WeightedEnergyEstimates} and \ref{ss:L^2ComparisonPrinciple}. Let us end this subsection by providing the following three remarks:
\begin{rem}[Relation between $\uu$ and $\ww$]
It is also worth to mention that
\[\ww = L_r \uu\]
in terms of the new differential operator and dependent variables \eqref{e:NewVar&Operator}.
\end{rem}
\begin{rem}[Boundary Condition for $\vv$]
Using the Biot-Savart formula $\eqref{e:VFAHEENV}_3$ and the definition \eqref{e:DefnofA} of $\mathcal{A}$, one may check that
\[\vv|_{r=0,1} \equiv 0.\]
\end{rem}
\begin{rem}[Sign Condition and Nonlinear Cancelation]
Using the new differential operator and dependent variables \eqref{e:NewVar&Operator}, we can express the sign condition \eqref{e:sign} and the
nonlinear cancelation \eqref{e:NonlinearCancelation} as \eqref{e:SignNV} and \eqref{e:NonlinearCancelationNV}
in Subsection \ref{ss:WeightedEnergyEstimates} respectively.
\end{rem}

%
%
\subsection{Weighted Energy Estimates}\label{ss:WeightedEnergyEstimates}
The aim of this subsection is to derive the a priori $H^s_L$ energy estimates for the quantity $\ww:=\frac{\w}{r}$. As a direct consequence, we will also obtain an $L^\infty$ control on $L_r \ww$ as well. More precisely, we will prove Proposition \ref{prop:APrioriEst} below.

The main idea of the $H^s_L$ energy estimates is to estimate the weighted energy \eqref{e:defnWH} below instead of the original $H^s_L$ norm \eqref{e:defnH}. Indeed, the weighted energy \eqref{e:defnWH}, which is well-defined under the sign condition \eqref{e:sign} (or equivalently, \eqref{e:SignNV} below), is well-chosen so that it can avoid the $x$-derivative loss, which is the structural difficulty that we mentioned in Subsection \ref{ss:MainResults}, by using the nonlinear cancelation \eqref{e:NonlinearCancelation} (or equivalently, \eqref{e:NonlinearCancelationNV} below), so the standard energy method works.

These a priori estimates only reply on the fact that if $(\uu,\vv,\ww)$ is a smooth solution to the vorticity system \eqref{e:VFAHEENV}, then the quantities $\uu$, $\vv$ and $\ww$ will satisfy the following three mathematical structures:
\begin{enumerate}[(i)]
\item the first order system: for $(t,x,r)\in (0,T)\times \T \times (0,1)$,
\begin{equation}\label{e:VFAHEENV3.2}
\left\{\begin{aligned}
\dt \ww + \uu \dx \ww + \vv L_r \ww & = 0 \\
\dx\uu + L_r\vv & = 0\\
\vv |_{r=0,1} & = 0\\
\ww |_{t=0} & = \ww_0 := L_r \uu_0;
\end{aligned}\right.
\end{equation}
\item the energy estimates for $\uu$ and $\vv$: for any integer $s\ge 0$, there exists a constant $C_s>0$ such that
\begin{equation}\label{e:EnergyEstforuu&vv}
\left\{\begin{aligned}
\|D_L^\alpha\uu\|_{L^2(rdrdx)} & \le C_s \|\ww\|_{H^s_L}& & \text{for any } |\alpha|\le s,\\
\|D_L^\alpha\vv\|_{L^2(rdrdx)} & \le C_s \|\ww\|_{H^s_L}& & \text{for any } |\alpha|\le s \text{ with }D_L^\alpha\neq\dx^s,\\
\end{aligned}\right.
\end{equation}
where $D_L^\alpha:=\dx^{\alpha_1} L_r^{\alpha_2}$, $L_r := \frac{1}{r} \partial_r$ and $\|\cdot\|_{H^s_L}$ is defined in \eqref{e:defnH}; and
\item the nonlinear cancelation: for any $k=0$, $1$, $2$, $\cdots$,
\begin{equation}\label{e:NonlinearCancelationNV}
\int_\T \int^1_0 \dx^k \vv \dx^k \ww \; rdrdx = 0.
\end{equation}
\end{enumerate}

\begin{rem}[Mathematical Structures]\label{rem:MathematicalStructures}
The above three mathematical structures can be easily verified as follows: let $(\uu,\vv,\ww)$ satisfy \eqref{e:VFAHEENV}.
\begin{enumerate}[(i)]
\item Indeed, $\eqref{e:VFAHEENV}_1$ and $\eqref{e:VFAHEENV}_4$ are exactly the same as $\eqref{e:VFAHEENV3.2}_1$ and $\eqref{e:VFAHEENV3.2}_4$, so it suffices to verify the incompressibility condition $\eqref{e:VFAHEENV3.2}_2$ and the boundary condition $\eqref{e:VFAHEENV3.2}_3$. However, they are just direct consequences of the Biot-Savart formulae $\eqref{e:VFAHEENV}_2$-$\eqref{e:VFAHEENV}_3$ and the definition \eqref{e:DefnofA} of $\mathcal{A}$;
\item Since $\ww = L_r \uu$ and $\dx\uu +L_r \vv =0$, the energy estimates \eqref{e:EnergyEstforuu&vv} just follow from the definition \eqref{e:defnH} of $H^s_L$ norm and the Poincar\'{e} inequality \eqref{e:PoincareIneqforL&mu};
\item Using $\ww = L_r \uu$, $\dx\uu +L_r \vv =0$, the integration by parts formula \eqref{e:IntnbyPartsforL_r&rdr} and the boundary condition $\eqref{e:VFAHEENV3.2}_3$, we have
    \[\int_\T \int^1_0 \dx^k \vv \dx^k \ww \; rdrdx = \frac{1}{2} \int_\T \int^1_0 \dx \left( \left| \dx^k \uu \right|^2 \right) \; rdrdx = 0,\]
     since $\uu$ is periodic in $x$.
\end{enumerate}
\end{rem}

Under the sign condition
\begin{equation}\label{e:SignNV}
L_r \ww := L_r^2 \uu \ge \sigma >0
\end{equation}
for some constant $\sigma$, the mathematical structures (i)-(iii) are sufficient to derive our a priori estimates for $C^\infty_L$ solutions,
where
\[\begin{split}
& C^\infty_L([0,T]\times\T\times [0,1]):= \{\ww:[0,T] \times \T \times [0,1]
\to \R;\\
& \quad\quad \dt^i \dx^j L^k_r \ww \text{ is continuous for any non-negative integers $i$, $j$ and $k$}\}.\end{split}\]
 More precisely, we will prove the following
\begin{prop}[A Priori Estimates]\label{prop:APrioriEst}
For any integer $s\ge 4$ and constant $\sigma >0$, let $\ww\in C([0,T];H^s_{L,\sigma})$, and $\uu$, $\vv$, $\ww\in C^\infty_L([0,T]\times\T\times [0,1])$. Assume that $\uu$, $\vv$ and $\ww$ satisfy system \eqref{e:VFAHEENV3.2}, energy estimates \eqref{e:EnergyEstforuu&vv}, and nonlinear cancelation \eqref{e:NonlinearCancelationNV}. Then we have the following estimates:
\begin{enumerate}[(i)]
\item (Weighted Energy Estimate) there exists a constant $C_{s,\sigma}>0$ such that
\begin{equation}\label{e:EnergyEstforww}
\|\ww(t)\|_{\tilde{H}^s_L}\le \|\ww_0\|_{\tilde{H}^s_L} + C_{s,\sigma} \int^t_0  \left\{1+\|\ww(\tau)\|_{\tilde{H}^s_L}\right\} \|\ww(\tau)\|^2_{\tilde{H}^s_L} \; d\tau,
\end{equation}
where the weighted energy $\|\cdot\|_{\tilde{H}^s_L}$ is defined by
\begin{equation}\label{e:defnWH}
\|\ww\|_{\tilde{H}^s_L}^2 := \left\|\frac{\dx^s \ww}{\sqrt{L_r \ww}}\right\|^2_{L^2 (rdrdx)}
+ \sum_{\substack{|\alpha|\le s\\\alpha_1 \ne s}} \|D^\alpha_L \ww\|^2_{L^2 (rdrdx)},
\end{equation}
and $D^\alpha_L := \dx^{\alpha_1} L_r^{\alpha_2}$;
\item ($L^\infty$ Estimate) there exist a time $\tilde{T}:= \tilde{T}(\sigma, \|\ww_0\|_{H^4_L})\in (0,T]$ and a constant $C_\sigma>0$ such
that for any $t\in [0,\tilde{T}]$, we have
\begin{equation}\label{e:LInftyEstforww}
\min_{\T\times [0,1]} L_r \ww_0 - C_\sigma \|\ww_0\|_{H^4_L}^2 t  \le L_r \ww (\tau,x,r)
 \le \max_{\T\times [0,1]} L_r \ww_0 + C_\sigma \|\ww_0\|_{H^4_L}^2 t,
\end{equation}
for all $(\tau,x,r)\in[0,t]\times\T\times [0,1]$.
\end{enumerate}
\end{prop}

\begin{rem}[Equivalent Energies] As long as $\ww \in H^s_{L,\sigma}$, the norm
\eqref{e:defnH} is equivalent to the weighted energy \eqref{e:defnWH} because of the $L^\infty$
boundedness of $L_r \ww$. Therefore, the weighted energy estimate \eqref{e:EnergyEstforww} provides a control
on $\|\ww\|_{H^s_L}$ as well.
\end{rem}

\begin{proof}[Proof of Proposition \ref{prop:APrioriEst}]
\hfil
\begin{enumerate}[(i)]
\item Inequality \eqref{e:EnergyEstforww} follows directly from the standard energy
method because one can apply the nonlinear cancelation \eqref{e:NonlinearCancelationNV} to avoid the $x$-derivative
loss, which is caused by $\vv = - \dx L_r^{-1} \uu$. The details are as follows.

Differentiating equation $\eqref{e:VFAHEENV3.2}_1$ with respect to $D^\alpha_L$, we have
\begin{equation}\label{e:DEnergyuu&vv}
\begin{split}
& \dt D^\alpha_L \ww + \uu \dx D^\alpha_L \ww + \vv L_r D^\alpha_L \ww \\
= & - \sum_{0<\beta \le \alpha} \binom{\alpha}{\beta} \left\{D^\beta_L \uu \dx D^{\alpha-\beta}_L \ww
+ D^\beta_L \vv L_r D^{\alpha-\beta}_L \ww \right\}.
\end{split}
\end{equation}

In the case that $|\alpha| \le s$ and $\alpha_1 \ne s$, since all $D^\beta_L \vv=- \dx^{\beta_1 + 1} L_r^{\beta_2 -1} \uu$ on the right hand side of
\eqref{e:DEnergyuu&vv} have no more than $s$ $x$-derivatives, we can apply the standard energy estimates, as
well as using the Sobolev inequality \eqref{e:sobolevinq} and energy estimates \eqref{e:EnergyEstforuu&vv}, to obtain
\begin{equation}\label{e:DtDww}
\dfrac{d}{dt} \|D^\alpha_L \ww\|^2_{L^2(r dr dx)} \le C_s \|\ww\|^3_{H^s_L}.
\end{equation}

In the case that $\alpha = (s,0)$, equation \eqref{e:DEnergyuu&vv} becomes
\begin{equation}\label{e:alphas0}
\begin{split}
& \dt\dx^s\ww + \uu \dx^{s+1} \ww + \vv \dx^s L_r \ww + \dx^s \vv L_r \ww \\
= & - \sum^s_{k=1} \binom{s}{k} \dx^k \uu \dx^{s-k+1} \ww - \sum^{s-1}_{k=1}
\binom{s}{k} \dx^k \vv \dx^{s-k} L_r \ww.
\end{split}
\end{equation}

Now, we are going to make use of the nonlinear cancelation \eqref{e:NonlinearCancelationNV} to
eliminate the problematic term $\dx^s \vv = - \dx^{s+1} L_r^{-1} \uu$, which has more than $s$ $x$-derivatives.
Multiplying \eqref{e:alphas0} by $\frac{\dx^s \ww}{L_r\ww}$, and then integrating over $\T \times (0,1)$ with
respect to $r dr dx$, we have, via using the evolution equation for $L_r \ww$,
\begin{equation}\label{e:evlnLrw}
\begin{split}
& \frac{1}{2} \dfrac{d}{dt} \iint \frac{|\dx^s \ww|^2}{L_r\ww} \;rdrdx + \iint  \dx^s \vv \dx^s \ww \;rdrdx\\
= & - \sum^s_{k=1} \binom{s}{k} \iint \frac{\dx^k \uu \dx^{s-k+1} \ww \dx^s \ww}{L_r \ww} \;rdrdx\\
& - \sum^{s-1}_{k=1} \binom{s}{k} \iint \frac{\dx^k \vv \dx^{s-k} L_r \ww \dx^s \ww}{L_r \ww} \;rdr dx\\
& + \frac{1}{2} \iint \frac{\ww\dx \ww - \dx \uu L_r \ww}{|L_r \ww|^2} |\dx^s \ww|^2 \;rdrdx.
\end{split}
\end{equation}
It follows from the nonlinear cancelation \eqref{e:NonlinearCancelationNV} that the second integral, which contains the problematic term $\dx^s \vv$, on the left hand side of \eqref{e:evlnLrw} is equal to 0. Furthermore, using the Sobolev inequality \eqref{e:sobolevinq},
energy estimates \eqref{e:EnergyEstforuu&vv} and the fact that $0 < \sigma \le L_r \ww \le \frac{1}{\sigma}$,
we know that the first two terms and the last term on the right hand side of \eqref{e:evlnLrw} are controlled
by $\frac{C_s}{\sigma} \|\ww\|^3_{H^s_L}$ and $\frac{C}{\sigma^2} \|\ww\|^2_{H^3_L} \|\dx^s \ww\|^2_{L^2(rdrdx)}$
respectively. Therefore, we have
\begin{equation}\label{e:dtdxsw}
\dfrac{d}{dt} \left\|\frac{\dx^s \ww}{\sqrt{L_r\ww}} \right\|^2_{L^2(rdrdx)} \le
C_{s, \sigma} \left( 1+ \|\ww\|_{H^s_L} \right) \|\ww\|^3_{H^s_L}.
\end{equation}
Using \eqref{e:DtDww} and \eqref{e:dtdxsw}, and summing up over $\alpha$, we finally obtain
\begin{equation}\label{e:dtww2}
\dfrac{d}{dt} \|\ww\|^2_{\tilde{H}^s_L} \le C_{s, \sigma} \left( 1+ \|\ww\|_{\tilde{H}^s_L} \right)
\|\ww\|^3_{\tilde{H}^s_L}
\end{equation}
since $\|\ww\|_{H^s_L} \le \frac{1}{\sqrt{\sigma}} \|\ww\|_{\tilde{H}^s_L}$. The weighted energy
estimate \eqref{e:EnergyEstforww} follows directly from \eqref{e:dtww2}.\\

\item The $L^\infty$ estimate \eqref{e:LInftyEstforww} is a direct consequence of the weighted energy estimate
\eqref{e:EnergyEstforww}. Indeed, inequality \eqref{e:EnergyEstforww} implies that there exist a time
$\tilde{T} := \tilde{T}(\sigma, \|\ww_0\|_{H^4_L}) \in (0,T]$ and a constant $C_\sigma>0$ such that
\begin{equation}\label{e:supww}
\sup_{[0,\tilde{T}]}\|\ww\|_{H^4_L} \le C_\sigma \|\ww_0\|_{H^4_L}.
\end{equation}
Hence, using the evolution equation for
$L_r \ww$, the Sobolev inequality \eqref{e:sobolevinq}, the energy estimates \eqref{e:EnergyEstforuu&vv} and the uniform bound \eqref{e:supww}, we have the following pointwise control: for any $0 \le t \le \tilde{T}$,
\begin{align*}
|\dt L_r \ww| & \le |\uu \dx L_r \ww| + |\vv L^2_r \ww| + |\ww \dx \ww| + |\dx\uu L_r \ww|\\
& \le C \left(\sup_{[0,\tilde{T}]} \|\ww\|_{H^4_L} \right)^2 \le C_\sigma \|\ww_0\|^2_{H^4_L},
\end{align*}
which implies the $L^\infty$ estimate \eqref{e:LInftyEstforww} by a direct integration.
\end{enumerate}
\end{proof}

%
%
\subsection{$L^2$ Comparison Principle and Its Applications to Uniqueness and Stability}\label{ss:L^2ComparisonPrinciple}
In this subsection we will first show a $L^2$ comparison principle for two solutions. Using this comparison principle, we will also obtain the uniqueness and stability to the axisymmetric hydrostatic Euler equations \eqref{e:AHEE}.

Let us begin with the $L^2$ comparison principle. The $L^2$ comparison principle is based on the following two properties for the difference of two solutions:
\begin{enumerate}[(i)]
\item Energy Estimate:
\begin{equation}\label{e:EnergyEstfor2Solns}
\|\uu_1-\uu_2\|_{L^2(r drdx)} \le C \|\ww_1-\ww_2\|_{L^2(rdrdx)};
\end{equation}
\item Nonlinear Cancelation:
\begin{equation}\label{e:NonlinearCancelationNVfor2Solns}
\int_\T \int^1_0 \left(\vv_1-\vv_2\right) \left(\ww_1-\ww_2\right) \; rdrdx = 0.
\end{equation}
\end{enumerate}
Using arguments similar to part (ii) and (iii) of Remark \ref{rem:MathematicalStructures}, one can verify that both \eqref{e:EnergyEstfor2Solns} and \eqref{e:NonlinearCancelationNVfor2Solns} hold for any two solutions $(\uu_1,\vv_1,\ww_1)$ and $(\uu_2,\vv_2,\ww_2)$ of the vorticity system \eqref{e:VFAHEENV}. Using energy estimate \eqref{e:EnergyEstfor2Solns} and nonlinear cancelation \eqref{e:NonlinearCancelationNVfor2Solns}, we will prove the following

\begin{prop}[$L^2$ Comparison Principle]\label{prop:L^2Comparison}
For $i=1$, $2$, let $\ww_i \in C([0,T];H^4_{L})\cap C^1([0,T];H^3_{L})$ such that $\uu_i$, $\vv_i$ and $\ww_i$ satisfy system \eqref{e:VFAHEENV3.2}, energy estimates \eqref{e:EnergyEstforuu&vv}, \eqref{e:EnergyEstfor2Solns}, and nonlinear cancelation \eqref{e:NonlinearCancelationNVfor2Solns}. If $\ww_2$ satisfies $0<\sigma \le L_r \ww_2 \le \frac{1}{\sigma}$, then we have the following $L^2$ comparison principle: there exists a constant $C_{\sigma,M, T}>0$ such that
\begin{equation}\label{e:L^2Comparison}
\|\ww_1-\ww_2\|_{C([0,T];L^2(rdrdx))} \le C_{\sigma,M,T} \|\left(\ww_1-\ww_2\right)|_{t=0}\|_{L^2(rdrdx)},
\end{equation}
where $M:=\max\left\{\sup_{[0,T]}\|\ww_1\|_{H^{4}_L}, \sup_{[0,T]}\|\ww_2\|_{H^{4}_L}\right\}$.
\end{prop}

\begin{proof}
Let $\tilde{\uu} := \uu_1 - \uu_2$, $\tilde{\vv} := \vv_1 - \vv_2$ and $\tilde{\ww} := \ww_1 - \ww_2$. Then $\tilde{\ww}$ satisfies
\begin{equation}\label{e:dttildeww}
\dt\tilde{\ww} + \uu_1 \dx \tilde{\ww} + \vv_1 L_r \tilde{\ww} + \tilde{\uu} \dx \ww_2 + \tilde{\vv} L_r \ww_2 = 0.
\end{equation}
Now, we are going to eliminate the problematic term $\tilde{\vv} = - \dx L^{-1}_r \tilde{\uu}$, which causes a loss of $x$-derivative,
by using the nonlinear cancelation \eqref{e:NonlinearCancelationNVfor2Solns}. Multiplying \eqref{e:dttildeww} by $\frac{\tilde{\ww}}{L_r \ww_2}$,
and then integrating over $\T \times (0,1)$ with respect to $r dr dx$, we have
\begin{equation}\label{e:eliminatev}
\begin{split}
& \frac{1}{2} \dfrac{d}{dt} \left\|\frac{\tilde{\ww}}{\sqrt{L_r \ww_2}} \right\|^2_{L^2 (r dr dx)} \\
= \; & \frac{1}{2} \iint \left\{(\dt + \uu_1 \dx + \vv_1 L_r) \left(\frac{1}{L_r \ww_2}\right)\right\} |\tilde{\ww}|^2 \;rdrdx\\
& - \iint \frac{\dx \ww_2}{L_r \ww_2} \tilde{\uu}\tilde{\ww} \;rdrdx - \iint \tilde{\vv}\tilde{\ww} \;rdrdx\\
\le \; & C_{\sigma, M} \left\| \frac{\tilde{\ww}}{\sqrt{L_r \ww_2}} \right\|^2_{L^2(r drdx)} + \|\tilde{u}\|^2_{L^2(r drdx)},
\end{split}
\end{equation}
where the last inequality is obtained by using the evolution equation for $L_r \ww_2$, the Sobolev inequality \eqref{e:sobolevinq}, energy estimates \eqref{e:EnergyEstforuu&vv} and the nonlinear cancelation \eqref{e:NonlinearCancelationNVfor2Solns}.

Using energy estimate \eqref{e:EnergyEstfor2Solns} and $0<\sigma \le L_r \ww_2 \le \frac{1}{\sigma}$, we have
\begin{equation}\label{e:tildeuu}
\|\tilde{\uu}\|^2_{L^2(r drdx)} \le C \|\tilde{\ww}\|^2_{L^2(rdrdx)} \le C_\sigma \left\| \frac{\tilde{\ww}}{\sqrt{L_r
\ww_2}} \right\|^2_{L^2 (r drdx)}.
\end{equation}

Combining \eqref{e:eliminatev} and \eqref{e:tildeuu}, we obtain
\begin{equation}\label{e:dttildewwover}
\dfrac{d}{dt} \left\| \frac{\tilde{\ww}}{\sqrt{L_r \ww}}\right\|_{L^2(r drdx)}
\le C_{\sigma, M} \left\| \frac{\tilde{\ww}}{\sqrt{L_r \ww_2}} \right\|_{L^2(r drdx)}.
\end{equation}
Applying the Gr\"{o}nwall's inequality to \eqref{e:dttildewwover}, we have
\[\left\|\frac{\tilde{\ww}}{\sqrt{L_r \ww_2}}\right\|_{L^2 (r drdx)}
\le C_{\sigma, M, T} \left\| \frac{\tilde{\ww}}{\sqrt{L_r \ww_2}} \Big|_{t=0} \right\|_{L^2(r drdx)},\]
which implies \eqref{e:L^2Comparison} because $0<\sigma \le L_r \ww_2 \le \frac{1}{\sigma}$.
\end{proof}

Using the $L^2$ comparison principle (i.e., Proposition \ref{prop:L^2Comparison}), we can immediately obtain the following corollary regarding the uniqueness and stability of the axisymmetric hydrostatic Euler equations \eqref{e:AHEE}.

\begin{cor}[Uniqueness and Stability of the Axisymmetric Hydrostatic Euler Equations \eqref{e:AHEE}]\label{cor:Uniqueness&StabilityAHEE}
For $i=1$, $2$, let $(\vec{u_i},p_i):= (u^x_i e_x + u^r_i e_r,p_i)$ be classical solutions to the axisymmetric hydrostatic Euler equations \eqref{e:AHEE} such that $\int^1_0 u^x_i \; r dr \equiv \lambda\equiv$ constant. For any integer $s\geq 4$ and constant $\sigma >0$, if $\ww_i := L_r u^x_i:= \frac{1}{r} \dr u^x_i \in C([0,T]; H^s_{L,\sigma}) \cap C^1 ([0,T]; H^{s-1}_L)$, then we have the following:
\begin{enumerate}[(i)]
\item (Stability) for any integer $s'\in [0,s)$, there exists a constant $C_{s,s',\sigma,T,M}>0$ such that the stability estimate \eqref{e:StabilityEsts} holds;
\item (Uniqueness) in addition, if both $(\vec{u_1},p_1)$ and $(\vec{u_2},p_2)$ satisfy the same initial data $u^x_0$, then $\vec{u_1}\equiv\vec{u_2}$ and $p_1\equiv p_2 + \tilde{p}$ for some function $\tilde{p}:=\tilde{p}(t)$.
\end{enumerate}
\end{cor}

\begin{proof}
Without loss of generality, we can assume that $\lambda =0$, otherwise, we can always fulfill this by applying a suitable Galilean transformation. See Remark \ref{rem:CompatibilityCondition} for the details.
\begin{enumerate}[(i)]
\item Since both $(\vec{u_1},p_1)$ and $(\vec{u_2},p_2)$ are regular enough, one may apply boundary condition $\eqref{e:AHEE}_3$, compatibility condition \eqref{e:CompatibilityCondition}, and elementary argument to show that both solutions satisfy the pole condition \eqref{e:PoleConditnfor u^r}. By the equivalence Lemma \ref{lem:Equivalence}, both $(u^x_1,u^r_1,\w_1)$ and $(u^x_2,u^r_2,\w_2)$ satisfy the vorticity system \eqref{e:VFAHEE}, where $\w_i:=\dr u^x_i$ for $i=1$, $2$. Equivalently, using the change of variables \eqref{e:NewVar&Operator}, we know that $(\uu_i,\vv_i,\ww_i):=(u^x_i,r u^r_i,\frac{\w_i}{r})$ satisfies the vorticity system \eqref{e:VFAHEENV} for $i=1$, $2$. Therefore, we can apply the $L^2$ comparison principle (i.e., Proposition \ref{prop:L^2Comparison}) to obtain estimate \eqref{e:L^2Comparison}. The stability estimate \eqref{e:StabilityEsts} follows directly from estimate \eqref{e:L^2Comparison} and the interpolation inequality \eqref{e:InterpolationIneq}.

\item Since both $(\vec{u_1},p_1)$ and $(\vec{u_2},p_2)$ satisfy the same initial data $u^x_0$, one can check that $\ww_1 |_{t=0} \equiv \ww_2 |_{t=0}$, where $\ww_i:=\frac{1}{r}\dr u^x_i$ for $i=1$, $2$. Thus, applying the stability estimate \eqref{e:StabilityEsts}, we obtain $\ww_1 \equiv \ww_2$. Using the Biot-Savart formulae $\eqref{e:VFAHEENV}_2$-$\eqref{e:VFAHEENV}_3$, we also have $\uu_1\equiv\uu_2$ and $\vv_1\equiv\vv_2$. In terms of the original variables, we obtain $(u^x_1,u^r_1,\w_1)\equiv (u^x_2,u^r_2,\w_2)$, and hence, by the equivalence Lemma \ref{lem:Equivalence}, $p_1\equiv p_2 + \tilde{p}$ for some function $\tilde{p}=\tilde{p}(t)$.
\end{enumerate}
\end{proof}

%
%
\section{Existence}\label{s:Existence}

The aim of this section is to provide two independent constructions of the solutions to the axisymmetric hydrostatic Euler equations \eqref{e:AHEE}. In Subsection \ref{ss:ExistenceviaApproximateSystem} we will first introduce an approximate scheme that keeps all a priori estimates derived in Section \ref{s:APrioriEstimates}, and then construct the solution as the limit of approximate systems. In Subsection \ref{ss:ExistenceviaReduction} we will construct the solution by a reduction argument.

%
%
\subsection{Existence via Approximate Scheme}\label{ss:ExistenceviaApproximateSystem}
In this subsection we will first introduce an approximate scheme. Using this approximate scheme, we will outline the proof of existence to the axisymmetric hydrostatic Euler equations \eqref{e:AHEE}.

Let us begin by introducing the approximate systems as follows: for any positive integer $N$, we consider, for $(t,x,r) \in (0,T) \times \T \times (0,1)$,
\begin{equation}\label{e:ApproxSysVFAHEENV}
\left\{\begin{aligned}
\dt \ww_N + \uu_N \dx \ww_N + \vv_N L_r \ww_N & = 0 \\
\uu_N & = - P_N L_r \mathcal{A} (\ww_N) \\
\vv_N & = P_N \dx \mathcal{A} (\ww_N) \\
\ww_N|_{t=0} & = \ww_0 := \frac{\w_0}{r}.
\end{aligned}\right.
\end{equation}
Here, the Dirichlet solver $\mathcal{A}$ is defined by \eqref{e:DefnofA}, and the projection operator $P_N$ is defined by
\begin{equation}\label{e:DefnofP_N}
\begin{aligned}
& P_N f(t,x,r) \\  := \; &  a_0 (t,r)  + \sum^N_{k=1} a_k (t,r) \left( \sqrt{2} \cos 2 k \pi x \right) + \sum^N_{k=1} b_k (t,r) \left( \sqrt{2} \sin 2 k \pi x \right),
\end{aligned}
\end{equation}
where the coefficients are given by
\[\left\{
\begin{aligned}
a_0 (t,r) &:= \int_{\mathbb{T}} f(t,x,r) \;dx \\
a_k (t,r) &:= \int_{\mathbb{T}} f(t,x,r) \left( \sqrt{2} \cos 2k\pi
x \right) \;dx && \mbox{ for all $k=1$, $2$, $\cdots$}\\
b_k (t,r) &:= \int_{\mathbb{T}} f(t,x,r) \left( \sqrt{2} \sin 2k\pi
x \right) \;dx && \mbox{ for all $k=1$, $2$, $\cdots$}.
\end{aligned}
\right.\]

The main advantage of the approximate system \eqref{e:ApproxSysVFAHEENV} is that for any fixed $N$, the system \eqref{e:ApproxSysVFAHEENV} does not have the loss of $x$-derivative (i.e., the structural difficulty mentioned in Subsection \ref{ss:MainResults}) because the projection operator $P_N$ regularizes $\vv_N$ in the $x$-direction. As a result, for any fixed $N$, one may construct the unique local-in-time solution to the approximate system \eqref{e:ApproxSysVFAHEENV} by classical methods. For example, one may first construct an iterative sequence of linearized solutions to \eqref{e:ApproxSysVFAHEENV}, and then prove the convergence of these linearized solutions by using the energy method. Since the details are standard, we leave this to the interested reader.

It is worth noting that if we only apply the standard energy method to solve the approximate system \eqref{e:ApproxSysVFAHEENV}, then the life-spans of the solutions may depend on $N$. This is due to the fact that the regularization effect of the projection $P_N$ becomes weaker and weaker for a larger and larger $N$. Therefore, in order to solve the approximate system \eqref{e:ApproxSysVFAHEENV} in a uniform (in $N$) life-span, one must derive a priori energy estimates without using the regularization effect of $P_N$. This can be done under the sign condition \ref{e:SignNV}. More precisely, one may check that the solution $(\uu_N,\vv_N,\ww_N)$ of \eqref{e:ApproxSysVFAHEENV} indeed satisfies the mathematical structures \eqref{e:VFAHEENV3.2}-\eqref{e:NonlinearCancelationNV}, and hence, we can also apply the Proposition \ref{prop:APrioriEst} to the solution $(\uu_N,\vv_N,\ww_N)$. As a result, weighted energy estimates \eqref{e:EnergyEstforww} and $L^\infty$ estimate \eqref{e:LInftyEstforww} also hold for $\ww_N$ as well. Since estimates \eqref{e:EnergyEstforww} and \eqref{e:LInftyEstforww} are uniform in $N$, one may apply the standard continuous induction argument to show that for any integer $s\ge 4$, $N=1$, $2$, $\cdots$ and $\sigma >0$, if $\ww_0\in H^s_{L,2 \sigma}$, there exist a uniform (in $N$) life-span $T>0$ and a solution $\ww_N\in C([0,T];H^s_{L,\sigma})\cap C^1([0,T];H^{s-1}_L)$ to the approximate system \eqref{e:ApproxSysVFAHEENV}. Furthermore, the solutions $\ww_N$ are uniformly bounded: there exists a constant $M>0$, depending on $s$, $\sigma$, $T$ and $\|\ww_0\|_{H^s_L}$ but not on $N$ nor $\ww_N$, such that
\begin{equation}\label{e:UniformEnergyEstforww_N}
\|\ww_N\|_{C([0,T];H^s_L)} \le M.
\end{equation}

Finally, in order to solve the vorticity system \eqref{e:VFAHEENV}, we have to prove the convergence of $\ww_N$ and the consistency of the limit. Regarding the convergence, one may follow a similar argument as in Subsection \ref{ss:L^2ComparisonPrinciple} to verify the $L^2$ convergence. However, for any two different approximate solutions, the nonlinear cancelation \eqref{e:NonlinearCancelationNVfor2Solns} does not hold in general. Indeed, we only have the almost nonlinear cancelation: for any integers $N_1>N_2$,
\begin{equation}\label{e:AlmostNonlinearCancelationNVfor2ApproxSolns}
\left|\int_\T \int^1_0 \left(\vv_{N_1}-\vv_{N_2}\right) \left(\ww_{N_1}-\ww_{N_2}\right) \; rdrdx\right| \le \frac{C_M}{N_2^2} \left\|\ww_{N_1}-\ww_{N_2}\right\|_{L^2(rdrdx)},
\end{equation}
where the constant $C_M$ depends on the constant $M$ mentioned in inequality \eqref{e:UniformEnergyEstforww_N}, but not on $N_1$ and $N_2$. As a result, if we follow the argument in the proof of Proposition \eqref{prop:L^2Comparison} and use the fact that $\left.\ww_{N_1}\right|_{t=0}\equiv\left.\ww_{N_2}\right|_{t=0}$, then we can show, for any integers $N_1>N_2$,
\begin{equation*}
\|\ww_{N_1}-\ww_{N_2}\|_{C([0,T];L^2(rdrdx))} \le  \frac{C_{\sigma,M,T}}{N_2^2}.
\end{equation*}
This proves the $C\left([0,T];L^2\right)$ convergence. By the interpolation inequality \eqref{e:InterpolationIneq}, we obtain the $C\left([0,T];H^{s'}_L\right)$ convergence of the approximate solutions $\ww_N$ as $N$ goes to $\infty$ for any integer $s'\in [0,s)$. Since we have the strong convergence, the consistency is obvious. Lastly, using the standard regularizing initial data argument, one may also construct $\ww\in C\left([0,T];H^{s}_L\right)$. Since the estimate \eqref{e:UniformEnergyEstforww_N} also holds for $\ww$, we can verify by using the evolution equation for $L_r \ww$ that if $\ww_0\in H^{s}_{L,2\sigma}$, then $\ww\in C\left([0,T];H^{s}_{L,\sigma}\right)$ for some $T>0$. This and the equivalence Lemma \ref{lem:Equivalence} show the existence stated in Theorem \ref{thm:Well-posed}.

Let us end this subsection by verifying the almost nonlinear cancelation \eqref{e:AlmostNonlinearCancelationNVfor2ApproxSolns} as follows:
\begin{proof}[Proof of Inequality \eqref{e:AlmostNonlinearCancelationNVfor2ApproxSolns}]
First of all, we can rewrite
\begin{equation}\label{e:Integral=I1+I2}
\int_\T \int^1_0 \left(\vv_{N_1}-\vv_{N_2}\right) \left(\ww_{N_1}-\ww_{N_2}\right) \; rdrdx = I_1 + I_2,
\end{equation}
where
\[\left\{
\begin{aligned}
I_1 &= \int_\T \int^1_0 \left(\vv_{N_1}-P_{N_1}\dx\mathcal{A}(\ww_{N_2})\right) \left(\ww_{N_1}-\ww_{N_2}\right) \; rdrdx \\
I_2 &= \int_\T \int^1_0 \left(P_{N_1}\dx\mathcal{A}(\ww_{N_2})-\vv_{N_2}\right) \left(\ww_{N_1}-\ww_{N_2}\right) \; rdrdx.
\end{aligned}
\right.\]
Using $\vv_{N_1}=P_{N_1}\dx\mathcal{A}(\ww_{N_1})$ and $\ww_{N_1}-\ww_{N_2}=-L_r^2\mathcal{A}\left(\ww_{N_1}-\ww_{N_2}\right)$, we can apply the integration by parts formula \eqref{e:IntnbyPartsforL_r&rdr} to the integral $I_1$, and obtain
\begin{equation}\label{e:I1Est}
\begin{aligned}
I_1 &= - \int_\T \int^1_0 P_{N_1}\dx\mathcal{A}\left(\ww_{N_1}-\ww_{N_2}\right) L_r^2 \mathcal{A}\left(\ww_{N_1}-\ww_{N_2}\right) \; rdrdx \\
&= \int_\T \int^1_0 P_{N_1}\dx L_r\mathcal{A}\left(\ww_{N_1}-\ww_{N_2}\right) L_r\mathcal{A}\left(\ww_{N_1}-\ww_{N_2}\right) \; rdrdx \\
&= \frac{1}{2} \int_\T \int^1_0 \dx \left| P_{N_1} L_r\mathcal{A}\left(\ww_{N_1}-\ww_{N_2}\right) \right|^2 \; rdrdx \\
&= 0.
\end{aligned}
\end{equation}
On the other hand, using $N_1>N_2$ and $\vv_{N_2}=P_{N_2}\dx\mathcal{A}(\ww_{N_2})$, we can estimate $I_2$ as follows:
\begin{equation}\label{e:I2Est}
\begin{aligned}
I_2 &= \int_\T \int^1_0 \left\{\left(P_{N_1}-P_{N_2}\right)\dx\mathcal{A}(\ww_{N_2})\right\} \left(\ww_{N_1}-\ww_{N_2}\right) \; rdrdx \\
&\le \|\left(P_{N_1}-P_{N_2}\right)\dx\mathcal{A}(\ww_{N_2})\|_{L^2(rdrdx)} \|\ww_{N_1}-\ww_{N_2}\|_{L^2(rdrdx)} \\
&\le \frac{C}{N_2^2}\|\left(P_{N_1}-P_{N_2}\right)\dx^3\mathcal{A}(\ww_{N_2})\|_{L^2(rdrdx)} \|\ww_{N_1}-\ww_{N_2}\|_{L^2(rdrdx)} \\
&\le \frac{C}{N_2^2}\|\ww_{N_2}\|_{H^3_L} \|\ww_{N_1}-\ww_{N_2}\|_{L^2(rdrdx)} \\
&\le \frac{C_M}{N_2^2}\|\ww_{N_1}-\ww_{N_2}\|_{L^2(rdrdx)},
\end{aligned}
\end{equation}
where the constant $C_M$ depends on the constant $M$ mentioned in inequality \eqref{e:UniformEnergyEstforww_N}. Combining \eqref{e:Integral=I1+I2}-\eqref{e:I2Est}, we prove inequality \eqref{e:AlmostNonlinearCancelationNVfor2ApproxSolns}.
\end{proof}

%
%
\subsection{Existence via Reduction}\label{ss:ExistenceviaReduction}
The aim of this subsection is to introduce a reduction argument which reduces the vorticity system \eqref{e:VFAHEENV} to be the vorticity system for the two-dimensional hydrostatic Euler equations. As a result, we can obtain the solutions to the axisymmetric hydrostatic Euler equations \eqref{e:AHEE} by using that to the two-dimensional hydrostatic Euler equations, whose $H^s$ well-posedness were shown in \cite{MW12} previously.

The main idea of the reduction argument is to use the cross-sectional area $a$ instead of the radius $r$ as an independent variable. To be more precise, let us begin with the following change of variable: define
\begin{equation}\label{e:Defnofa}
a:=\frac{1}{2}r^2,
\end{equation}
which is equivalent (up to a factor of $2\pi$) to the area of a two-dimensional disk of radius $r$. It follows from \eqref{e:Defnofa} that
\begin{equation}\label{e:da&L_r&rdr}
\da = \frac{1}{r}\dr =: L_r \quad\text{and}\quad da = rdr.
\end{equation}
Thus, using the change of variable \eqref{e:Defnofa}, we can rewrite the vorticity system \eqref{e:VFAHEENV} as the following system:
for $(t,x,a) \in (0,T) \times \T \times (0,\frac{1}{2})$,
\begin{equation}\label{e:VFAHEEinx&a}
\left\{\begin{aligned}
\dt \ww + \uu \dx \ww + \vv \da \ww & = 0 \\
\uu & = - \da \mathcal{A} (\ww) \\
\vv & = \dx \mathcal{A} (\ww) \\
\ww|_{t=0} & = \ww_0,
\end{aligned}\right.
\end{equation}
where the Dirichlet solver $\mathcal{A}$ is given by
\begin{equation}\label{e:DefnofAina}
\mathcal{A}(f)(t,x,a) := -\frac{1}{2} \int^{\frac{1}{2}}_0 \{|a - \tilde{a}| -a - \tilde{a}  + 4 a \tilde{a}\} f(t,x,\tilde{a})
\;d\tilde{a}.
\end{equation}
In other words, $\mathcal{A}(f)$ defined in \eqref{e:DefnofAina} is the unique solution to
\[\left\{\begin{aligned}
-\da^2 \mathcal{A} (f) & = f &\text{in }\T\times(0,\frac{1}{2})\\
\mathcal{A} (f) |_{a=0,\frac{1}{2}} & = 0.
\end{aligned}\right.\]

Rescaling the interval $(0,\frac{1}{2})$ to be $(0,1)$ and viewing the variable $a$ as the length instead of area, one can recognize that the vorticity system \eqref{e:VFAHEEinx&a} is exactly the same as the vorticity system\footnote{See system (2.1) in \cite{MW12} for instance.} for the two-dimensional hydrostatic Euler equations. Furthermore, by \eqref{e:da&L_r&rdr}, the sign condition \eqref{e:SignNV} can be written as
\[
\da \ww = \da^2 \uu \geq \sigma >0,
\]
which can also be seen as the local Rayleigh condition\footnote{See inequality (2.2) in \cite{MW12} for instance.} for the two-dimensional hydrostatic Euler equations.

As a result, it follows from Theorem 2.5 in \cite{MW12} that for any integer $s\geq 4$ and constant $\sigma\in (0,\frac{1}{2})$, if the initial data $\ww_0\in H^s_{2\sigma} (\T\times (0,\frac{1}{2}))$, then the vorticity system \eqref{e:VFAHEEinx&a} has a unique solution $\ww\in C([0,T];H^s_\sigma (\T\times (0,\frac{1}{2})))\cap C^1([0,T];H^{s-1} (\T\times (0,\frac{1}{2})))$, where the life-span $T$ depends on $\|\ww_0\|_{H^s}$, $s$ and $\sigma$ only. Here, the function space $H^s_\sigma (\T\times (0,\frac{1}{2}))$ is defined by
\begin{equation*}
H^s_\sigma (\T\times (0,\frac{1}{2})) := \left\{ \ww\in H^s (\T\times (0,\frac{1}{2})) ;\; 0<\sigma\leq \da\ww\leq\frac{1}{\sigma} \right\},
\end{equation*}
and $H^s (\T\times (0,\frac{1}{2}))$ is the standard $H^s$ space defined on $\T\times (0,\frac{1}{2})$ equipped with the standard $H^s$ norm:
\begin{equation}\label{e:DefnofH^sNorm}
\|\ww\|_{H^s}^2 := \sum_{|\alpha |\leq s} \int_0^\frac{1}{2} \int_\T \left| \dx^{\alpha_1} \da^{\alpha_2} \ww \right|^2 \; dxda.
\end{equation}
Using \eqref{e:da&L_r&rdr}, one can verify that the standard $H^s$ norm \eqref{e:DefnofH^sNorm} with respect to $x$ and $a$ is equivalent to the $H^s_L$ norm defined in \eqref{e:defnH}, and hence, the  function spaces $H^s (\T\times (0,\frac{1}{2}))$ and $H^s_\sigma (\T\times (0,\frac{1}{2}))$ are equivalent to $H^s_L$ defined in \eqref{e:DefnofH^s_L} and $H^s_{L,\sigma}$ defined in \eqref{e:DefnofH^s_Lsigma} respectively. Therefore, the existence stated in Theorem \ref{thm:Well-posed} follows directly from Theorem 2.5 in \cite{MW12} and the equivalence Lemma \ref{lem:Equivalence}.

%
%
\section{Mathematical Justification of the Formal Derivation}\label{s:MathematicalJustification}
The aim of this section is to provide a rigorous mathematical justification of the formal derivation of the hydrostatic Euler equations \eqref{e:AHEE}. In other words, we will verify the formal derivation introduced in Subsection \ref{ss:FormalDerivation} by proving Theorem \ref{thm:MathJustification}.

Let us begin by rewriting the axisymmetric rescaled Euler equations \eqref{e:AREE} in terms of the differential operator $L_r:=\frac{1}{r}\dr$ and new dependent variables $\uu_\e$, $\vv_\e$ and $\ww_\e$, which are analogous to the variables $\uu$, $\vv$ and $\ww$ introduced in \eqref{e:NewVar&Operator} previously. Define
\begin{equation}\label{e:NewVarforAREE}
\begin{aligned}
\uu_\e &:=u^x_\e,& \vv_\e &:= r u^r_\e,\text{ and}& \ww_\e &:= L_r \uu_\e - \frac{\e^2}{r^2} \dx\vv_\e.
\end{aligned}
\end{equation}
Using \eqref{e:NewVarforAREE}, we can express the axisymmetric rescaled Euler equations \eqref{e:AREE} as the following equations: for $(t,x,r) \in (0,T) \times
\T \times (0,1)$,
\begin{equation}\label{e:AREENV}
\left\{\begin{aligned}
\dt \uu_\e + \uu_\e \dx \uu_\e + \vv_\e L_r \uu_\e & = - \dx p_\e \\
\e^2 \left\{\dt \left( \frac{\vv_\e}{r} \right) + \uu_\e \dx \left( \frac{\vv_\e}{r} \right) + \vv_\e L_r \left( \frac{\vv_\e}{r} \right) \right\} & = - \dr p_\e \\
\dx \uu_\e + L_r \vv_\e & = 0 \\
\vv_\e|_{r=0,1} & = 0\\
\uu_\e|_{t=0} & = \uu_{\e 0}:=u^x_{\e 0}.
\end{aligned}\right.
\end{equation}
In terms of the change of variables \eqref{e:NewVarforAREE}, the corresponding vorticity system, called the axisymmetric rescaled vorticity system, becomes: for $(t,x,r)\in (0,T)\times \T \times (0,1)$,
\begin{equation}\label{e:VFAREENV}
\left\{\begin{aligned}
\dt \ww_\e + \uu_\e \dx \ww_\e + \vv_\e L_r \ww_\e & = 0 \\
\dx\uu_\e + L_r\vv_\e & = 0\\
\vv_\e |_{r=0,1} & = 0\\
\ww_\e|_{t=0} & = \ww_{\e 0} := L_r \uu_{\e 0}.
\end{aligned}\right.
\end{equation}

Now, the question becomes as follows: for any solution $(\vec{u},p):=(u^x e_x +u^r e_r,p)$ to the axisymmetric hydrostatic Euler equations \eqref{e:AHEE}, does there exist a sequence $\{(\uu_\e,\vv_\e,\ww_\e)\}_{\e>0}$ such that $(\uu_\e,\vv_\e,\ww_\e)$ converges to $(\uu,\vv,\ww):=(u^x,r u^r,L_r u^x)$ as $\e$ goes to $0^+$? Under the sign condition \eqref{e:SignNV}, the answer is affirmative in the following $L^2$ sense:
\begin{prop}[Mathematical Justification of the Formal Derivation]\label{prop:MathJustificationNV}
Under the hypotheses of Theorem \ref{thm:MathJustification}, if there exist constants $C_0>0$ and
$\beta \in(0,4]$ such that
\[
\left.\int_\T \int^1_0 \left\{|\uu_\e-\uu|^2 + \e^2\left|\frac{\vv_\e-\vv}{r}\right|^2 + |\ww_\e-\ww|^2\right\}\;rdrdx\right|_{t=0} \le C_0 \e^\beta,
\]
then for all $t \in [0,T]$,
\begin{equation}\label{e:L^2MathJustificationNV}
\int_\T \int^1_0 \left\{|\uu_\e-\uu|^2+\e^2\left|\frac{\vv_\e-\vv}{r}\right|^2+|\ww_\e-\ww|^2\right\} \;rdrdx
\le \widetilde{C} \e^\beta,
\end{equation}
where the constant $\widetilde{C}$ depends only on $\sigma$, $\uu$, $\vv$, $\ww$, $C_0$ and $T$, but not on
$\e$ nor $(\uu_\e,\vv_\e,\ww_\e)$.
\end{prop}

It is worth noting that Theorem \ref{thm:MathJustification} is a direct consequence of Proposition \ref{prop:MathJustificationNV}. Following the approach in \cite{Bre03}, we will prove Proposition \ref{prop:MathJustificationNV} by using the entropy method. Under appropriate modifications, our proof is similar to the proof in \cite{Bre03}, in which the formal derivation of the two-dimensional hydrostatic Euler equations was rigorously justified. These modifications are necessary because the structure of the axisymmetric rescaled Euler equations \eqref{e:AREENV} is different\footnote{More precisely, the axisymmetric rescaled Euler equations \eqref{e:AREENV} is not equivalent to the two-dimensional rescaled Euler equations even if we apply the reduction argument stated in Subsection \ref{ss:ExistenceviaReduction}.} than that of the two-dimensional rescaled Euler equations\footnote{See equations (9.2) in \cite{MW12} for instance.}.

\begin{proof}[Proof of Proposition \ref{prop:MathJustificationNV}]
In this proof we will first define a convex functional that is equivalent to the left hand side of \eqref{e:L^2MathJustificationNV}. Then we will derive a growth rate control on this convex functional. The main difficulty of this proof is that the terms involving $\vv_\e -\vv$ may cause a loss of one $\epsilon$, which is corresponding to the $x$-derivative loss caused by the hydrostatic limit. However, this $\epsilon$ loss can be avoided by a well-chosen convex functional due to the nonlinear cancelation.

First of all, we are going to define the convex functional by using the following convex (in $\ww$) function:

\begin{lem}[Convex Function]\label{lem:ConvexFunction}

For any given $\uu$ and $\ww$ satisfying the sign condition \eqref{e:SignNV}, there exist a constant $\kappa$, and a smooth\footnote{Here, ``smooth" means that for any fixed $(t,x)\in [0,T]\times\T$, the functions $F$, $\dt F$ and $\dx F$ are twice continuously differentiable with respect to $\ww$, and the partial derivatives $\dww^2 F$, $\dt\dww^2 F$, $\dx\dww F$ and $\dx\dww^2 F$ are bounded in $[0,T]\times\T\times\R$.} and strongly convex\footnote{That is, there exists a constant $c>0$ such that $\dww^2 F \geq c$.} (in $\ww$) function $F:[0,T]\times\T\times\R\to\R$ such that for all $(t,x,r)\in [0,T]\times\T\times [0,1]$,
\begin{equation}\label{e:dww^2F>1}
\dww^2 F \left( t,x,\ww(t,x,r) \right) = \frac{\uu (t,x,r) - \kappa}{L_r \ww (t,x,r)} \geq 1.
\end{equation}
Here, the $C^3$ norm of $F$ depends on $\uu$, $\ww$, and the $\sigma$ stated in the sign condition \eqref{e:SignNV}, but the constant $\kappa$ depends on $\uu$ and $\ww$ only.
\end{lem}

It is worth noting that both $\kappa$ and $F$ stated in Lemma \ref{lem:ConvexFunction} are independent of $\e$ and $(\uu_\e,\vv_\e,\ww_\e)$. Furthermore, the constant $1$ stated in \eqref{e:dww^2F>1} is not a crucial value: the proof of Proposition \ref{prop:MathJustificationNV} will also work if it is replaced by any other positive constant. Assuming Lemma \ref{lem:ConvexFunction}, which will be shown at the end of this section, for the moment, we can define the convex functional by using the convex function $F$ as follows.

For any given $(\uu,\vv,\ww)$, we define the following convex functional
\begin{equation}\label{e:DefnofConvexFunctnal}
L_{\e}(t):=L_{k,\e}(t)+L_{c,\e}(t),
\end{equation}
where the kinetic energy part $L_{k,\e}$ and the convex part $L_{c,\e}$ are given by
\begin{equation}\label{e:DefnofConvexFunctnal(KE&ConvexParts)}
\left\{
\begin{aligned}
L_{k,\e}(t) &:= \frac{1}{2}\int_\T \int_0^1 \left\{ |\uu_\e - \uu|^2 + \e^2 \left| \frac{\vv_\e - \vv}{r} \right|^2 \right\} \;rdrdx \\
L_{c,\e}(t) &:= \frac{1}{2}\int_\T \int_0^1 \{ F(t,x,\ww_\e) - F(t,x,\ww) \\
& \qquad\qquad\qquad\qquad - \dww F(t,x,\ww) (\ww_\e - \ww) \} \;rdrdx.
\end{aligned}
\right.
\end{equation}

Using the smoothness and convexity of $F$, one may verify that the convex functional $L_{\e}$ is equivalent to the left hand side of \eqref{e:L^2MathJustificationNV}, so in order to prove Proposition \ref{prop:MathJustificationNV}, it suffices to show
\[
L_{\e} (t) \lesssim \e^\beta
\]
for $\beta\in (0,4]$ provided that $L_{\e}(0)$ does initially.

A direct computation, which we leave to the interested reader, yields
\begin{lem}\label{lem:TimeDerivatives}
\begin{equation}\label{e:TimeDerivatives}
\left\{\begin{aligned}
\dfrac{d}{dt} L_{k,\e} &= I_1 + I_2 + Y + R \\
\dfrac{d}{dt} L_{c,\e} &= I_3 + I_4 + I_5 + X,
\end{aligned}\right.
\end{equation}
where
\begin{align*}
I_1 &:= -\frac{1}{2} \int_\T \int_0^1 \dx \uu \left\{ |\uu_\e - \uu|^2 + \e^2 \left| \frac{\vv_\e - \vv}{r} \right|^2 \right\} \;rdrdx, \\
I_2 &:= -\e^2 \int_\T \int_0^1 \dx \left(\frac{\vv}{r}\right) \left( \uu_\e - \uu \right) \left( \frac{\vv_\e - \vv}{r} \right) \;rdrdx \\
& \qquad \qquad \qquad \qquad -\e^2 \int_\T \int_0^1 \dr \left(\frac{\vv}{r}\right) \left| \frac{\vv_\e - \vv}{r} \right|^2 \;rdrdx, \\
I_3 &:= - \int_\T \int_0^1 \dx \ww \dww^2 F(t,x,\ww) \left( \uu_\e - \uu \right) \left( \ww_\e - \ww \right) \;rdrdx, \\
I_4 &:= \int_\T \int_0^1 \{ \dt F(t,x,\ww_\e) - \dt F(t,x,\ww) - \dt\dww F(t,x,\ww) \left( \ww_\e - \ww \right) \} \;rdrdx, \\
I_5 &:= \int_\T \int_0^1 \uu_\e \{ \dx F(t,x,\ww_\e) - \dx F(t,x,\ww) - \dx\dww F(t,x,\ww) \left( \ww_\e - \ww \right) \} \;rdrdx, \\
X &:= - \int_\T \int_0^1 \dww^2 F(t,x,\ww) L_r\ww \left( \vv_\e - \vv \right) \left( \ww_\e - \ww \right) \;rdrdx, \\
Y &:= \int_\T \int_0^1 \uu \left( \vv_\e - \vv \right) \left( \ww_\e - \ww \right) \;rdrdx, \\
R &:= -\e^2 \int_\T \int_0^1 \left( \dt + \vv L_r \right) \left( \frac{\vv}{r} \right) \left( \frac{\vv_\e - \vv}{r} \right) \;rdrdx.
\end{align*}
\end{lem}
Applying Lemma \ref{lem:TimeDerivatives}, we have
\begin{equation}\label{e:dtL_e}
\dfrac{d}{dt} L_{\e} = \dfrac{d}{dt} L_{k,\e} + \dfrac{d}{dt} L_{c,\e} = \sum_{i=1}^5 I_i + Z + R,
\end{equation}
where $Z:=X+Y$, that is,
\[
Z = - \int_\T \int_0^1 \left\{ \dww^2 F(t,x,\ww) L_r\ww - \uu \right\} \left( \vv_\e - \vv \right) \left( \ww_\e - \ww \right) \;rdrdx.
\]
It is easy to check that for all $i=1$, $2$, $\cdots$, $5$,
\begin{equation}
|I_i| \leq C_{\sigma,\uu,\vv,\ww} \; L_\e,
\end{equation}
so we only have to control $Z$ and $R$. Indeed, both $Z$ and $R$ have the factor $\vv_\e - \vv$, which is problematic because it may create a loss of one $\epsilon$. Therefore, in order to obtain good estimates for $Z$ and $R$, one should eliminate the problematic factor $\vv_\e - \vv$. This can be done by making use of the well-chosen convex function $F$ and the integration by parts argument as follows.

First of all, to deal with $Z$, we apply \eqref{e:dww^2F>1} to rewrite
\[
Z = \kappa \int_\T \int_0^1 \left( \vv_\e - \vv \right) \left( \ww_\e - \ww \right) \;rdrdx.
\]
Using the identity
\[\ww_\e-\ww=L_r (\uu_\e-\uu) - \frac{\e^2}{r^2}\dx\vv_\e,\]
the integration by parts formula \eqref{e:IntnbyPartsforL_r&rdr}, and the incompressibility conditions \eqref{e:IncompressibilityConditnindx&Lr} and $\eqref{e:VFAREENV}_2$, we have
\begin{equation}\label{e:NonlinearCancelationforZ}
\begin{aligned}
Z &= \kappa \int_\T \int_0^1 \left( \vv_\e - \vv \right) L_r \left( \uu_\e - \uu \right) \;rdrdx \\
& \qquad - \kappa \e^2 \int_\T \int_0^1 \left( \frac{\vv_\e - \vv}{r} \right) \dx \left( \frac{\vv_\e - \vv}{r} \right) \;rdrdx \\
& \qquad - \kappa \e^2 \int_\T \int_0^1 \dx \left( \frac{\vv}{r} \right)  \left( \frac{\vv_\e - \vv}{r} \right) \;rdrdx \\
&= \frac{\kappa}{2} \int_\T \int_0^1 \dx \left\{ \left|\uu_\e - \uu\right|^2 - \e^2 \left|\frac{\vv_\e - \vv}{r}\right|^2 \right\} \;rdrdx \\
& \qquad - \kappa \e^2 \int_\T \int_0^1 \dx \left( \frac{\vv}{r} \right) \left( \frac{\vv_\e - \vv}{r} \right) \;rdrdx \\
&= - \kappa \e^2 \int_\T \int_0^1 \dx \left( \frac{\vv}{r} \right) \left( \frac{\vv_\e - \vv}{r} \right) \;rdrdx
\end{aligned}
\end{equation}
since all $\uu_\e$, $\vv_\e$, $\uu$ and $\vv$ are periodic in $x$. It is worth noting that in \eqref{e:NonlinearCancelationforZ} we applied the nonlinear cancelation argument to eliminate the leading order term. Combining \eqref{e:dtL_e}-\eqref{e:NonlinearCancelationforZ}, we have
\begin{equation}\label{e:dtL_e<CL_e+tildeR}
\dfrac{d}{dt} L_{\e} = \sum_{i=1}^5 I_i + Z + R \leq C_{\sigma,\uu,\vv,\ww} \; L_\e + \tilde{R},
\end{equation}
where $\tilde{R}$ is defined by
\[
\tilde{R} := - \e^2 \int_\T \int_0^1 \left( \dt + \kappa\dx + \vv L_r \right) \left( \frac{\vv}{r} \right) \left( \frac{\vv_\e - \vv}{r} \right) \;rdrdx.
\]
Finally, we can estimate $\tilde{R}$ by a simple integration by parts argument. Define
\[
q(t,x,r):=\int_0^r \left. \left( \dt + \kappa\dx + \vv L_r \right) \left( \frac{\vv}{r} \right) \right|_{(t,x,r)=(t,x,\tilde{r})} \;d\tilde{r}.
\]
Then using the integration by parts formula \eqref{e:IntnbyPartsforL_r&rdr}, and the incompressibility conditions \eqref{e:IncompressibilityConditnindx&Lr} and $\eqref{e:VFAREENV}_2$, we have
\begin{equation}\label{e:EstimatefortildeR}
\begin{aligned}
\tilde{R} &= - \e^2 \int_\T \int_0^1 L_r q \left( \vv_\e - \vv \right) \;rdrdx = \e^2 \int_\T \int_0^1 q L_r \left( \vv_\e - \vv \right) \;rdrdx \\
&= - \e^2 \int_\T \int_0^1 q \dx \left( \uu_\e - \uu \right) \;rdrdx = \e^2 \int_\T \int_0^1 \dx q  \left( \uu_\e - \uu \right) \;rdrdx \\
&\leq C_q L_\e + \e^4.
\end{aligned}
\end{equation}
Combining \eqref{e:dtL_e<CL_e+tildeR} and \eqref{e:EstimatefortildeR}, we obtain
\[
\dfrac{d}{dt} L_{\e} \leq C_{\sigma,\uu,\vv,\ww} \; L_\e + \e^4.
\]
Applying the Gr\"{o}nwall's inequality, we have
\[
L_{\e} (t) \leq \left\{ L_{\e} (0) + \e^4 t \right\} e^{C_{\sigma,\uu,\vv,\ww} \; t},
\]
which implies \eqref{e:L^2MathJustificationNV} for $\beta\in (0,4]$.
\end{proof}

In order to complete the proof of Proposition \ref{prop:MathJustificationNV}, we will show Lemma \ref{lem:ConvexFunction} as follows:
\begin{proof}[Proof of Lemma \ref{lem:ConvexFunction}]
Following a similar argument in \cite{Bre03}, we will construct the convex function $F$ as follows. Since $\frac{1}{r}\dr\ww = L_r \ww \geq \sigma >0$, we know that for any fixed $(t,x)$, the mapping $r\mapsto \ww(t,x,r)$ is strictly increasing. Therefore, let
\[\Omega :=\{ (t,x,\tilde{\ww})\in [0,T]\times\T\times\R; \; \ww (t,x,0)\leq\tilde{\ww}\leq \ww(t,x,1)\},\]
and we can define a function $R:\Omega\to\R$ such that
\[
R(t,x,\tilde{\ww}):=r \quad\text{if}\quad\ww(t,x,r)=\tilde{\ww}.
\]
In other words, for any $(t,x,r)\in [0,T]\times\T\times [0,1]$ and $(t,x,\tilde{\ww})\in\Omega$, we have
\[
r=R(t,x,\ww (t,x,r)) \quad\text{and}\quad \tilde{\ww}=\ww(t,x,R(t,x,\tilde{\ww})).
\]
Now, for any fixed constant $\kappa\in\R$, we can define, via using this mapping $R$,
\[
G(t,x,\tilde{\ww}):= \frac{\uu(t,x,R(t,x,\tilde{\ww}))-\kappa}{L_r \ww (t,x,R(t,x,\tilde{\ww}))}
\]
for all $(t,x,\tilde{\ww})\in\Omega$. We can further define $F$ by integrating $G$ twice with respect to $\tilde{\ww}$, and extending $F$ smoothly in the variable $\tilde{\ww}$. The extension here is not unique and not important in the proof of Proposition \ref{prop:MathJustificationNV}. To obtain inequality \eqref{e:dww^2F>1}, $\kappa$ must be chosen appropriately, for example, we can choose $\kappa := \inf \uu - \sup L_r\ww$. Here, the choice of $\kappa$ is not unique as well.
\end{proof}

%
%
\section{Formation of Singularity}\label{s:FormationofSingularity}
In this section we will prove the blowup result (i.e., Theorem \ref{thm:FormationofSingularity}) for the axisymmetric hydrostatic Euler equations \eqref{e:AHEE}. The proof is based on the reduction argument introduced in Subsection \ref{ss:ExistenceviaReduction} and the recent blowup result in \cite{Won15}.

\begin{proof}[Proof of Theorem \ref{thm:FormationofSingularity}]
Using the \emph{new} differential operator and dependent variables \eqref{e:NewVar&Operator}, we can rewrite the axisymmetric hydrostatic Euler equations \eqref{e:AHEE} as follows: for $(t,x,r) \in (0,T) \times \R \times (0,1)$,
\begin{equation}\label{e:AHEENV}
\left\{\begin{aligned}
\dt \uu + \uu \dx \uu + \vv L_r \uu & = - \dx p \\
\dx \uu + L_r \vv & = 0 \\
\vv|_{r=0,1} & = 0\\
\uu|_{t=0} & = \uu_0.
\end{aligned}\right.
\end{equation}
Similar to Subsection \ref{ss:ExistenceviaReduction}, we can apply the change of variables
\begin{equation*}
\tilde{a}:=r^2 \quad\text{and}\quad \tilde{\vv}:=2\vv
\end{equation*}
to rewrite \eqref{e:AHEENV} as the following equations: for $(t,x,\tilde{a}) \in (0,T) \times \R \times (0,1)$,
\begin{equation}\label{e:AHEEinx&a}
\left\{\begin{aligned}
\dt \uu + \uu \dx \uu + \tilde{\vv} \partial_{\tilde{a}} \uu & = - \dx p \\
\dx \uu + \partial_{\tilde{a}} \tilde{\vv} & = 0 \\
\tilde{\vv}|_{\tilde{a}=0,1} & = 0\\
\uu|_{t=0} & = \uu_0,
\end{aligned}\right.
\end{equation}
which is exactly the same as the two-dimensional hydrostatic Euler equations if we see $\tilde{a}$ as the length instead of area.

Similarly, under the same change of variables, one may check that the hypothesis \eqref{e:BlowupHypothesis} becomes
\begin{equation}\label{e:BlowupHypothesisinx&a}
\left\{
\begin{aligned}
\uu_0 (\hat{x},\tilde{a}) &\equiv \widehat{\uu} \quad\text{for all }\tilde{a}\in [0,1]\\
\text{either}\quad\dx \partial_{\tilde{a}} \uu_0 (\hat{x},0) &= 0 \quad\text{or}\quad\dx \partial_{\tilde{a}} \uu_0 (\hat{x},1) = 0\\
\dx \partial_{\tilde{a}}^2 \uu_0(\hat{x},\tilde{a}) &< 0 \quad\text{for all }\tilde{a}\in (0,1),
\end{aligned}\right.
\end{equation}
where $\widehat{\uu}:=\widehat{u^x}$ is the given constant horizontal velocity. Without loss of generality, we can further assume that $\eqref{e:BlowupHypothesisinx&a}_2$ is just
\begin{equation}\label{e:BlowupHypothesisinx&aWLOG}
\dx \partial_{\tilde{a}} \uu_0 (\hat{x},0) = 0;
\end{equation}
otherwise, we can apply the change of variables $\tilde{\tilde{a}}:=1-\tilde{a}$ and $\tilde{\tilde{\vv}}:=-\tilde{\vv}$ as well.

Since the initial data $\uu_0$ satisfies the hypotheses $\eqref{e:BlowupHypothesisinx&a}_1$, $\eqref{e:BlowupHypothesisinx&a}_3$ and \eqref{e:BlowupHypothesisinx&aWLOG}, by Theorem 2.1 in \cite{Won15}, there exists a finite time $T>0$ such that if a solution $(\uu,\tilde{\vv},p)$ to the two-dimensional hydrostatic Euler equations \eqref{e:AHEEinx&a} remains smooth in the time interval $[0,T)$, then
\begin{equation}\label{e:FiniteTimeBlowupuudxuudxp}
\lim_{t\to T^-} \max\left\{ \|\uu(t)\|_{L^\infty}, \|\dx \uu(t)\|_{L^\infty}, \|\dx p(t)\|_{L^\infty} \right\}=\infty.
\end{equation}
Indeed, \eqref{e:FiniteTimeBlowupuudxuudxp} is equivalent to \eqref{e:FiniteTimeBlowupu^xdxu^xdxp} under the above change of variables, so Theorem \ref{thm:FormationofSingularity} is just a direct consequence of Theorem 2.1 in \cite{Won15}.
\end{proof}


\appendix

%
%
\section{Basic Properties for the Axisymmetric Hydrostatic Euler Equations}\label{app:BasicPropertiesforAHEE}

In this appendix we will study the basic properties for the solutions of the axisymmetric hydrostatic Euler equations \eqref{e:AHEE}. More precisely, we will show that the average horizontal velocity is conserved, and the axisymmetric hydrostatic Euler equations \eqref{e:AHEE} and its vorticity system \eqref{e:VFAHEE} are equivalent.

\begin{lem}[Conservation of the Average Horizontal Velocity]\label{lem:ConservationofAveu^x}
Let $(\vec{u},p):= (u^x e_x + u^r e_r,p)$ be a classical solution to the axisymmetric hydrostatic Euler equations \eqref{e:AHEE} in the periodic physical domain $\T \times (0,1)$ such that the pole condition \eqref{e:PoleConditnfor u^r} holds. Then
\begin{equation}\label{e:ConservationofAveu^x}
\int^1_0 u^x \; rdr \equiv \int^1_0 u^x_0 \; rdr.
\end{equation}
\end{lem}
\begin{proof}
Integrating $\eqref{e:AHEE}_1$ and $\eqref{e:AHEE}_2$ with respect to $rdr$ over $(0,1)$, we have, after applying integration by parts, the boundary condition $\eqref{e:AHEE}_3$ and the pole condition \eqref{e:PoleConditnfor u^r},
\begin{equation}\label{e:AveEqts}
\left\{\begin{aligned}
\dt \int_0^1 u^x \; rdr + \dx \int_0^1 \left| u^x \right|^2 \; rdr &= - \frac{1}{2} \dx p \\
\dx \int_0^1 u^x \; rdr &= 0.
\end{aligned}\right.
\end{equation}
It follows directly from $\eqref{e:AveEqts}_2$ that $\int_0^1 u^x \; rdr$ is independent of $x$. Thus, integrating $\eqref{e:AveEqts}_1$ with respect to $x$ over $\T$ and using the $x$-periodicity of the functions $u^x$ and $p$, we obtain
\[
\dt \int_0^1 u^x \; rdr = 0,
\]
which implies the identity \eqref{e:ConservationofAveu^x}.
\end{proof}

\begin{lem}[Equivalence Lemma]\label{lem:Equivalence}
We have the following:
\begin{enumerate}[(i)]
\item Let $(u^x,u^r,\w)$ be a classical solution to the vorticity system \eqref{e:VFAHEE}. Define
      \begin{equation}\label{e:Formulaforp}
      p := - 2 \int^1_0 \left|u^x\right|^2 \; rdr + \tilde{p},
      \end{equation}
      where $\tilde{p}:=\tilde{p} (t)$ is any arbitrary function of $t$. Then $(\vec{u},p):= (u^x e_x + u^r e_r,p)$ solves the axisymmetric hydrostatic Euler equations \eqref{e:AHEE}, and satisfies the pole condition \eqref{e:PoleConditnfor u^r} and the compatibility condition \eqref{e:CompatibilityCondition}.
\item Let the initial horizontal velocity $u^x_0$ satisfy the compatibility condition \eqref{e:CompatibilityCondition} and $(\vec{u},p):= (u^x e_x + u^r e_r,p)$ be a classical solution to the axisymmetric hydrostatic Euler equations \eqref{e:AHEE}. If $u^r$ satisfies the pole condition \eqref{e:PoleConditnfor u^r} and $\w:=\dr u^x \in C^1 ([0,T]\times\T\times[0,1])$, then $(u^x,u^r,\w)$ solves the vorticity system \eqref{e:VFAHEE}, and $u^x$ satisfies the compatibility condition \eqref{e:CompatibilityCondition} for all later times $t>0$.
\end{enumerate}
\end{lem}
\begin{proof}
The proof is just a direct checking and will be outlined as follows.
\begin{enumerate}[(i)]
\item It follows directly from the Biot-Savart formulae $\eqref{e:VFAHEE}_2$-$\eqref{e:VFAHEE}_3$ and the definition \eqref{e:DefnofA} of $\mathcal{A}$ that $\w= \dr u^x$, and $(\vec{u},p)$ satisfies the incompressibility condition $\eqref{e:AHEE}_2$, the boundary condition $\eqref{e:AHEE}_3$, the pole condition \eqref{e:PoleConditnfor u^r}, as well as the compatibility condition \eqref{e:CompatibilityCondition}. It remains to show that $(\vec{u},p)$ also satisfies the momentum equation $\eqref{e:AHEE}_1$ and the initial condition $\eqref{e:AHEE}_4$.

    Substituting $\w= \dr u^x$ into the vorticity equation $\eqref{e:VFAHEE}_1$ and using the incompressibility condition $\eqref{e:AHEE}_2$, we have
    \[
    \dr\left\{ \dt u^x + u^x \dx u^x + u^r \dr u^x \right\} = 0,
    \]
    and hence, a direct integration yields, via using the boundary condition $\eqref{e:AHEE}_3$,
    \begin{equation}\label{e:LHSofAHEE_1}
    \dt u^x + u^x \dx u^x + u^r \dr u^x = q,
    \end{equation}
    where $q:=\left\{ \dt u^x + u^x \dx u^x \right\}|_{r=1}$ is a function of $t$ and $x$ only. To see that $q$ is actually equal to $-\dx p$, we integrate \eqref{e:LHSofAHEE_1} with respect to $rdr$ over $(0,1)$, and obtain
    \begin{equation}\label{e:Formulaforq}
    \frac{1}{2}q = \int_0^1 u^x \dx u^x + u^r \dr u^x \; rdr = \dx \int_0^1 \left|u^x\right|^2\; rdr = - \frac{1}{2}\dx p.
    \end{equation}
    Here, we applied \eqref{e:CompatibilityCondition} in the first equality. In the second equality, we applied integration by parts, $\eqref{e:AHEE}_2$-$\eqref{e:AHEE}_3$ and \eqref{e:PoleConditnfor u^r}. The third equality follows directly from formula \eqref{e:Formulaforp}. Combining \eqref{e:LHSofAHEE_1} and \eqref{e:Formulaforq}, we verify the momentum equation $\eqref{e:AHEE}_1$.

    Finally, using the Biot-Savart formula $\eqref{e:VFAHEE}_2$ and $\w= \dr u^x$, we have
    \[
    \left. u^x \right|_{t=0} = -\frac{1}{r}\dr\mathcal{A}\left(\frac{\w_0}{r}\right)= -\frac{1}{r}\dr\mathcal{A}\left(\frac{1}{r}\dr u^x_0\right) = u^x_0,
    \]
    where the last equality holds under the compatibility condition \eqref{e:CompatibilityCondition}. This verifies the initial condition $\eqref{e:AHEE}_4$ and completes the proof of part (i).
\item First of all, a direct differentiation of initial condition $\eqref{e:AHEE}_4$ with respect to $r$ yields the initial condition $\eqref{e:VFAHEE}_4$. Next, differentiating the momentum equation $\eqref{e:AHEE}_1$ with respect to $r$ and using the incompressibility condition $\eqref{e:AHEE}_2$, we obtain the vorticity equation $\eqref{e:VFAHEE}_1$.

    According to Lemma \ref{lem:ConservationofAveu^x}, the compatibility condition \eqref{e:CompatibilityCondition} holds for all later times $t>0$ because it does initially. Finally, the Biot-Savart formulae $\eqref{e:VFAHEE}_2$-$\eqref{e:VFAHEE}_3$ can be verified by using the standard Green's function technique because $u^x$ and $u^r$ are the unique solutions to
    \[
    \left\{\begin{aligned}
    \frac{1}{r}\dr u^x &= \frac{\w}{r} \\
    \int_0^1 u^x \;rdr &\equiv 0
    \end{aligned}\right.\quad\text{and}\quad
    \left\{\begin{aligned}
    -\left(\frac{1}{r}\dr\right)^2 \left(r u^r\right) &= \frac{\dx\w}{r} \\
    \left. r u^r \right|_{r=0,1} &= 0
    \end{aligned}\right.
    \]
    respectively.
\end{enumerate}
\end{proof}

%
%
\section{Properties for the Operator $L_r$}\label{app:Operator L_r}
The aim of this appendix is to provide basic properties for the operator $L_r$ and its corresponding measure $rdr$. In particular, we will discuss the fundamental theorem of calculus, the integration by parts formula, the Poincar\'{e} type inequality, the Sobolev type inequality and the interpolation inequality.

Let us begin by stating without proof the following calculus facts for $L_r$ and $rdr$:

\begin{prop}[Calculus Facts for $L_r$ and $rdr$]\label{prop:CalFactforL_r}
Define the operator $L_r:=\frac{1}{r}\dr$. Then we have
\begin{enumerate}[(i)]
\item (Fundamental Theorem of Calculus) for any continuously differentiable function $f$, for any $a$, $b\in [0,1]$,
\begin{equation}\label{e:FTCforL_r&rdr}
f(b)=f(a)+\int^b_a L_r f \; rdr;
\end{equation}
\item (Integration by Parts Formula) for any continuously differentiable functions $f$ and $g$, for any $a$, $b\in [0,1]$,
\begin{equation}\label{e:IntnbyPartsforL_r&rdr}
\int^b_a f L_r g \; rdr = \left[ fg\right]^b_{r=a} - \int^b_a g L_r f \; rdr.
\end{equation}
\end{enumerate}
\end{prop}

The proof of Proposition \ref{prop:CalFactforL_r} is elementary, so we leave this to the reader. As long as an operator $L$ and its corresponding measure $\mu$ satisfy the fundamental theorem of calculus, the following Poincar\'{e} type inequality follows immediately. In particular, inequality \eqref{e:PoincareIneqforL&mu} below holds when $L:=L_r$ and $d\mu := rdr$.

\begin{prop}[Poincar\'{e} Type Inequality]\label{prop:PoincareIneq}
Let $L$ be an operator and $\mu$ be a measure on $[0,1]$. Assume that the operator $L$ and its corresponding measure $\mu$ satisfy the fundamental theorem of Calculus: for any function $f$ in the domain of $L$, for any $a$, $b\in [0,1]$,
\begin{equation}\label{e:FTCforL&mu}
f(b)=f(a)+\int^b_a L f \; d\mu.
\end{equation}
If there exists a point $r_0\in [0,1]$ such that $f(r_0)=0$, then the following Poincar\'{e} type inequality holds:
\begin{equation}\label{e:PoincareIneqforL&mu}
\int^1_0 |f|^2 \; d\mu \le \mu \left( [0,1] \right)^2 \int^1_0 |Lf|^2 \; d\mu.
\end{equation}
\end{prop}
\begin{proof}
It follows from $f(r_0)=0$ and \eqref{e:FTCforL&mu} that for any $r\in [0,1]$,
\[
f(r) = \int^r_{r_0} L f \; d\mu,
\]
and hence, by the Cauchy-Schwarz inequality,
\[
|f(r)|^2 \le \left| \int^1_{0} |L f| \; d\mu \right|^2 \le \mu \left( [0,1] \right) \int^1_0 |Lf|^2 \; d\mu.
\]
Integrating the above inequality with respect to the measure $\mu$ over $[0,1]$, we prove inequality \eqref{e:PoincareIneqforL&mu}.
\end{proof}

Furthermore, we also have the following Sobolev type and interpolation inequalities:

\begin{prop}[Sobolev Type Inequality]\label{prop:SobolevIneq}
Define $L_r:=\frac{1}{r}\dr$. There exists a universal constant $C>0$ such that
\begin{equation}\label{e:sobolevinq}
\|f\|_{L^\infty (rdrdx)}\leq C \left\{ \|f\|_{L^2 (rdrdx)} + \|\dx f\|_{L^2 (rdrdx)} + \|L_r^2 f\|_{L^2 (rdrdx)} \right\},
\end{equation}
for any function $f:\T\times (0,1) \to\R$.
\end{prop}
\begin{proof}[Outline of the Proof]
Let $a:=\frac{1}{2} r^2$. Then $\partial_a = \frac{1}{r}\dr = L_r$ and $dadx$ = $rdrdx$. Therefore, inequality \eqref{e:sobolevinq} is equivalent to
\begin{equation}\label{e:sobolevinqindadx}
\|f\|_{L^\infty (dadx)}\leq C \left\{ \|f\|_{L^2 (dadx)} + \|\dx f\|_{L^2 (dadx)} + \|\partial_a^2 f\|_{L^2 (dadx)} \right\},
\end{equation}
which can be shown by elementary methods. For instance, see Lemma B.2 in \cite{MWCPAM} for the proof of \eqref{e:sobolevinqindadx} in a similar domain.
\end{proof}

\begin{prop}[Interpolation Inequality]
For any integers $s>0$ and $s' \in [0,s)$, there exists a constant $C_{s,s'}>0$ such that
\begin{equation}\label{e:InterpolationIneq}
\|f\|_{H^{s'}_L} \leq C_{s,s'}  \|f\|_{H^{s}_L}^\frac{s'}{s} \|f\|_{L^2(rdrdx)}^{1-\frac{s'}{s}},
\end{equation}
where the norm $\|\cdot\|_{H^s_L}$ is defined in \eqref{e:defnH}.
\end{prop}

\begin{proof}[Outline of the Proof]
Similar to the proof of Proposition \ref{prop:SobolevIneq}, we define $a:=\frac{1}{2} r^2$. Since $\partial_a = \frac{1}{r}\dr = L_r$ and $dadx$ = $rdrdx$, we have
\begin{align*}
\|f\|_{H^s_L}^2 : & = \sum_{|\alpha| \le s} \int_\T \int^1_0 |\dx^{\alpha_1} L_r^{\alpha_2} f|^2 \; rdrdx \\
& = \sum_{|\alpha| \le s} \int_\T \int^{\frac{1}{2}}_0 |\dx^{\alpha_1} \partial_a^{\alpha_2} f|^2 \; dadx =:\|f\|_{H^s(dadx)}^2,
\end{align*}
where $\|\cdot\|_{H^s(dadx)}$ is the standard $H^s$ norm with respect to $x$ and $a$. As a result, the interpolation inequality \eqref{e:InterpolationIneq} is equivalent to
\[
\|f\|_{H^{s'}(dadx)} \leq C_{s,s'}  \|f\|_{H^{s}(dadx)}^\frac{s'}{s} \|f\|_{L^2(dadx)}^{1-\frac{s'}{s}},
\]
which is just the the standard interpolation inequality. Thus, \eqref{e:InterpolationIneq} holds.
\end{proof}

%
%
\section{Properties for Three Dimensional Smooth Vector Fields}\label{app:3DSmoothVectorFields}

The aim of this appendix is to provide a brief explanation why the unique solution obtained in Theorem \eqref{thm:Well-posed} is a smooth vector field in three spatial dimensions. We will first quote the characterization result in \cite{LW09}, and then, apply this to verify the regularity of our solution.

In \cite{LW09}, Liu and Wang gave a characterization of smoothness of axisymmetric and divergence-free vector fields in terms of the pole conditions (see $\eqref{e:Conditionforpsi}_2$ below), which are compatibility conditions at the axis of symmetry in a certain sense. Using the notation in this paper, we can restate\footnote{The way that we state Proposition \ref{prop:Liu&Wang} is slightly different than the original way in \cite{LW09}. For the original statement, please refer to \cite{LW09}.} a special case of Lemma 2 of \cite{LW09} regarding axisymmetric vector fields without swirl as follows:

\begin{prop}[Special Case of Lemma 2 in \cite{LW09}]\label{prop:Liu&Wang}
Denote $r:=\sqrt{y^2+z^2}$. Let $e_x$ and $e_r$ be the unit vectors in the horizontal (i.e., $x$) and radial (i.e., $r$) directions respectively. As a three-dimensional vector field, $\vec{u}:=u^x(x,r) e_x + u^r(x,r) e_r$ is divergence-free and belongs to $C^k(\R^3)$ if and only if there exists a scalar-valued function $\psi :=\psi (x,r)\in C^{k+1}(\R\times\overline{\R^+})$ such that
\begin{equation}\label{e:Conditionforpsi}
\left\{
\begin{aligned}
u^x & = L_r (r\psi), \quad\quad  u^r = -\dx\psi, \\
\left.\dr^{2m}\psi\right|_{r=0^+} &\equiv 0 \quad\quad  \text{for all }0\leq 2m \leq k,
\end{aligned}
\right.
\end{equation}
where $L_r:=\frac{1}{r}\dr$.
\end{prop}

Now, we are going to apply Proposition \ref{prop:Liu&Wang} to provide a better understanding of our solutions to the axisymmetric hydrostatic Euler equations \eqref{e:AHEE}. Up to a Galilean transformation (c.f. Remark \ref{rem:CompatibilityCondition}), the axisymmetric velocity field $\vec{u}$ obtained in Theorem \ref{thm:Well-posed} belongs to
\begin{align*}
\mathcal{V}^s_{as}:= \bigg\{ & \vec{u}:=u^x(x,r) e_x + u^r(x,r) e_r ; \; \int_0^1 u^x \; rdr \equiv 0, \; \left. u^r\right|_{r=1}=0,\\
& \quad\quad \dx u^x + L_r (r u^r) = 0, \;\text{and } \ww:=L_r u^x \in H^s_L  \bigg\},
\end{align*}
where $L_r :=\frac{1}{r}\dr$, and the function space $H^s_L$ is defined in \eqref{e:DefnofH^s_L}. Indeed, as a three-dimensional vector field, our solution $\vec{u}$ is $C^{s-2}$ at the axis of symmetry (i.e., $r=0$) as well. More precisely, we have the following
\begin{prop}[Regularity of $\mathcal{V}^s_{as}$ Vector Fields]
Let $r:=\sqrt{y^2+z^2}$ and $\Omega:=\left\{(x,y,z);\;x\in\T,\; y^2+z^2<1\right\}$. For any integer $s\geq 2$, \begin{equation}\label{e:V^s_asinC^s-2}
\mathcal{V}^s_{as}\subseteq C^{s-2} (\Omega).
\end{equation}
\end{prop}
\begin{proof}[Outline of the Proof] Define $\phi(x,r):=-\int_r^1 u^x (x,\tilde{r}) \; \tilde{r}d\tilde{r}$. Then we have
\[
u^x = L_r \phi, \quad r u^r = -\dx \phi, \quad\text{and}\quad \left.\phi\right|_{r=0,1}\equiv 0.
\]
In other words, $\frac{\phi}{r}$ is equal to the scalar-valued function $\psi$ stated in Proposition \ref{prop:Liu&Wang}. Furthermore, using the Sobolev inequality \eqref{e:sobolevinq}, the Poincar\'{e} inequality \eqref{e:PoincareIneqforL&mu} and the standard density argument, one may verify that $\phi$ is $C^{s-1}$ outside any neighborhood of $r=0$, so is $\psi=\frac{\phi}{r}$.

According to Proposition \ref{prop:Liu&Wang}, it remains to verify the pole condition $\eqref{e:Conditionforpsi}_2$. It follows from a direct computation and L'H\^{o}pital's rule that
\begin{align*}
\lim_{r\to 0^+} \dr^j \psi & = \lim_{r\to 0^+} \dr^j \left(\frac{\phi}{r}\right) = \lim_{r\to 0^+} \frac{(-1)^j j!}{r^{j+1}}\sum_{i=0}^j \frac{(-1)^i r^i \dr^i \phi}{i!} \\
& = \lim_{r\to 0^+} \frac{(-1)^j j!}{(j+1)r^{j}}\dr\left(\sum_{i=0}^j \frac{(-1)^i r^i \dr^i \phi}{i!}\right) = \frac{1}{j+1} \lim_{r\to 0^+} \dr^{j+1} \phi. \\
\end{align*}
Therefore, it suffices to show that $\lim_{r\to 0^+} \dr^{2m+1} \phi=0$ for all non-negative integer $m\leq \frac{s}{2}-1$. Since $\dr = r L_r$, a direct computation yields
\begin{equation}\label{e:dr^2m+1phi}
\dr^{2m+1} \phi = \sum_{i=m+1}^{2m+1} a_{m,i} \; r^{2i-2m-1} L_r^i\phi,
\end{equation}
for some constants $a_{m,i}$. Using the Sobolev inequality \eqref{e:sobolevinq} and the Poincar\'{e} inequality \eqref{e:PoincareIneqforL&mu}, one may check that the absolute values of all $L_r^i\phi$ on the right hand side of \eqref{e:dr^2m+1phi} are bounded by $\|\ww\|_{H^s_L}$, and hence, passing to the limit $r\to 0^+$ in \eqref{e:dr^2m+1phi}, we obtain $\lim_{r\to 0^+} \dr^{2m+1} \phi=0$ for all non-negative integer $m\leq \frac{s}{2}-1$. This completes the proof and verifies \eqref{e:V^s_asinC^s-2}.
\end{proof}

%
%


\begin{thebibliography}{10}

\bibitem{Bre99}
Brenier, Yann.
\newblock Homogeneous hydrostatic flows with convex velocity profiles.
\newblock {\em Nonlinearity}, 12 (1999), no. 3, 495–-512.

\bibitem{Bre03}
Brenier, Yann.
\newblock Remarks on the derivation of the hydrostatic Euler equations
\newblock {\em Bull. Sci. Math.}, 127 (2003), no. 7, 585--595.

\bibitem{Bre08}
Brenier, Yann.
\newblock Generalized solutions and hydrostatic approximation of the Euler equations.
\newblock {\em Phys. D},  237 (2008), no. 14-17, 1982-–1988.

\bibitem{CINT15}
Cao, Chongsheng; Ibrahim, Slim; Nakanishi, Kenji and Titi, Edriss S..
\newblock Finite-time blowup for the inviscid primitive equations of oceanic and atmospheric dynamics
\newblock {\em Comm. Math. Phys.}, 337 (2015), no. 2, 473–-482.

\bibitem{Gre99}
Grenier, Emmanuel.
\newblock On the derivation of homogeneous hydrostatic equations.
\newblock {\em M2AN Math. Model. Numer. Anal.}, 33 (1999), no. 5, 965–-970.

\bibitem{Gre00}
Grenier, Emmanuel.
\newblock On the nonlinear instability of Euler and Prandtl equations.
\newblock {\em Comm. Pure Appl. Math.}, 53 (2000), no. 9, 1067-–1091.

\bibitem{KMVW14}
Kukavica, Igor; Masmoudi, Nader; Vicol, Vlad and Wong, Tak Kwong.
\newblock On the local well-posedness of the Prandtl and hydrostatic Euler equations with multiple monotonicity regions.
\newblock {\em SIAM J. Math. Anal.}, 46 (2014), no. 6, 3865–-3890.

\bibitem{KTVZ11}
Kukavica, Igor; Temam, Roger; Vicol, Vlad C. and Ziane, Mohammed.
\newblock Local existence and uniqueness for the hydrostatic Euler equations on a bounded domain.
\newblock {\em J. Differential Equations}, 250 (2011), no. 3, 1719-–1746.

\bibitem{Lio96}
Lions, Pierre-Louis.
\newblock {\em Mathematical topics in fluid mechanics. Vol. 1. Incompressible models.}
\newblock {\em  Oxford Lecture Series in Mathematics and its Applications}, 3, Oxford Science Publications. The Clarendon Press, Oxford University Press, New York, 1996.

\bibitem{LW09}
Liu, Jian-Guo and Wang, Wei-Cheng.
\newblock Characterization and regularity for axisymmetric solenoidal vector fields with application to Navier-Stokes equation.
\newblock {\em SIAM J. Math. Anal.}, 41 (2009), no. 5, 1825–-1850.

\bibitem{Maj86}
Majda, Andrew.
\newblock Vorticity and the mathematical theory of incompressible fluid flow. Frontiers of the mathematical sciences: 1985 (New York, 1985).
\newblock {\em Comm. Pure Appl. Math.}, 39 (1986), no. S, suppl., S187–-S220.

\bibitem{MW12}
Masmoudi, Nader and Wong, Tak Kwong.
\newblock On the $H^s$ theory of hydrostatic Euler equations.
\newblock {\em Arch. Ration. Mech. Anal.}, 204 (2012), no. 1, 231--271.

\bibitem{MWCPAM}
Masmoudi, Nader and Wong, Tak Kwong.
\newblock Local-in-time existence and uniqueness of solutions to the Prandtl equations by energy methods.
\newblock Preprint, arXiv:1206.3629, and to appear in {\em Comm. Pure Appl. Math.}.

\bibitem{Ren09}
Renardy, Michael.
\newblock Ill-posedness of the hydrostatic Euler and Navier-Stokes equations.
\newblock {\em Arch. Ration. Mech. Anal.}, 194 (2009), no. 3, 877-–886.

\bibitem{SY94}
Shirota, Taira and Yanagisawa, Taku.
\newblock Note on global existence for axially symmetric solutions of the Euler system.
\newblock {\em  Proc. Japan Acad. Ser. A Math. Sci.}, 70 (1994), no. 10, 299–-304.

\bibitem{UI68}
Ukhovskii, M. R. and Iudovich, V. I..
\newblock Axially symmetric flows of ideal and viscous fluids filling the whole space.
\newblock {\em  J. Appl. Math. Mech.}, 32 (1968), 52–-61.

\bibitem{Won15}
Wong, Tak Kwong.
\newblock Blowup of solutions of the hydrostatic Euler equations
\newblock {\em Proc. Amer. Math. Soc.}, 143 (2015), no. 3, 1119–-1125.

\end{thebibliography}

\end{document}